\documentclass[11pt,reqno]{amsart}

\usepackage[utf8]{inputenc}
\usepackage[english]{babel}
\usepackage[a4paper]{geometry}

\usepackage{amssymb,amsmath,amsthm}
\usepackage[v2,tips]{xy}
\usepackage{xcolor}
\usepackage{hyperref}

\newcommand\blfootnote[1]{%
  \begingroup
  \renewcommand\thefootnote{}\footnote{#1}%
  \addtocounter{footnote}{-1}%
  \endgroup
}

\theoremstyle{plain}
\newtheorem{proposition}{Proposition}[section]
\newtheorem{theorem}[proposition]{Theorem}
\newtheorem{corollary}[proposition]{Corollary}
\newtheorem{lemma}[proposition]{Lemma}
\newtheorem{claim}[proposition]{Claim}
\newtheorem{definition}[proposition]{Definition}
\newtheorem{example}[proposition]{Example}

\theoremstyle{definition}
\newtheorem{remark}[proposition]{Remark}
\newtheorem{ack}[proposition]{Acknowledgement}

\numberwithin{equation}{section}

\DeclareMathOperator*{\Ran}{\mathrm{Ran}}
\DeclareMathOperator*{\Ker}{\mathrm{Ker}}
\DeclareMathOperator*{\inv}{\mathrm{inv}}
\DeclareMathOperator*{\asy}{\mathrm{Asy}}
\newcommand{\support}{\operatorname{supp}}
\newcommand{\spanning}{\operatorname{span}}

\newcommand{\mc}{\mathcal}

\title[A purely infinite Cuntz-like Banach $*$-algebra]{A purely infinite Cuntz-like Banach $*$-algebra with no purely infinite ultrapowers}
\author{Matthew Daws and Bence Horv\'{a}th}

\newcommand{\Addresses}{{
  \bigskip
  \footnotesize
  
  \textit{Addresses}: \\
  Matthew Daws: \textsc{Jeremiah Horrocks Institute, University of Central Lancashire, Preston, PR1~2HE, United Kingdom} \\
  Bence Horv\'{a}th: \textsc{Institute of Mathematics, Czech Academy of Sciences, \v{Z}itn\'a 25, 115 67 Prague 1, Czech Republic}\par\nopagebreak
  \textit{E-mail addresses}: \\ 
  Matthew Daws: \texttt{matt.daws@cantab.net, mdaws@uclan.ac.uk} \\
  Bence Horv\'{a}th: \texttt{horvath@math.cas.cz, hotvath@gmail.com}
}}

\begin{document}

\begin{abstract}
We continue our investigation, from \cite{dh}, of the ring-theoretic infiniteness properties of ultrapowers of Banach algebras, studying in this paper the notion of being purely infinite.  It is well known that a $C^*$-algebra is purely infinite if and only if any of its ultrapowers are.  We find examples of Banach algebras, as algebras of operators on Banach spaces, which do have purely infinite ultrapowers.  Our main contribution is the construction of a ``Cuntz-like'' Banach $*$-algebra which is purely infinite, but whose ultrapowers are not even simple, and hence not purely infinite. This algebra is a naturally occurring analogue of the Cuntz algebra, and of the $L^p$-analogues introduced by Phillips.  However, our proof of being purely infinite is combinatorial, but direct, and so differs from existing proofs. We show that there are non-zero traces on our algebra, which in particular implies that our algebra is not isomorphic to any of the $L^p$-analogues of the Cuntz algebra.
\end{abstract}

\maketitle

\subjclass{2020}{\textit{ Mathematics Subject Classification.} 46M07, 46H10, 46H15 (primary); 43A20 (secondary)

\blfootnote{}

\section{Introduction and preliminaries}

\subsection{Introduction}
We continue our study of infiniteness properties of Banach algebras, and how these interact with reduced products, in the continuous model theory sense, which we initiated in \cite{dh}.  Recall that an idempotent $p$ in an algebra $\mc A$ is \emph{infinite} if it is (algebraically Murray--von Neumann) equivalent to a proper sub-idempotent of itself.  One prominent property which we did not study in \cite{dh} is that of being \emph{purely infinite}, which for simple rings could be defined by saying that every left ideal contains an infinite idempotent.  We discuss this notion, and the literature surrounding it, in Section~\ref{sec:pi} below.  This definition is equivalent, for a unital Banach algebra, to $\mc A$ not being $\mathbb C$, and that for $a\in\mc A$ non-zero there are $b,c\in\mc A$ with $bac=1$.  This generalises the definition for $C^*$-algebras.

As a purely infinite Banach algebra must be simple, the asymptotic sequence algebra of $\mc A$ is never purely infinite, see Remark~\ref{r: asy_not_pi} below.  We thus focus on ultrapowers in this paper.  As in \cite{dh}, and perhaps not surprisingly from the perspective of continuous model theory, we find that an ultrapower $(\mc A)_{\mc U}$ is purely infinite if and only if it satisfies a ``metric'' form of the definition, where we have some sort of norm control. That purely infinite $C^*$-algebras have purely infinite ultrapowers follows from such norm control always being available.

In \cite{dh} we found examples of Banach algebras which did, and did not, have suitable forms of norm control.  Our major tool was to look at weighted semigroup algebras, where the weight allowed us to vary the norm control which we obtained.  Surprisingly, in this paper we have no need to consider weights.  Thus our examples are somewhat more ``natural'', and indeed, in showing that our principal example does not have simple (hence purely infinite) ultrapowers, we proceed in a somewhat indirect way, and avoid directly computing norms in the ultrapower.

The structure of the paper is as follows.  In the remainder of the introduction, we provide a more detailed introduction to purely infinite algebras, and recall the ultrapower construction.  In Section~\ref{sec:norm_control} we define a suitable ``quantified'' definition of being purely infinite, and show that this does indeed capture when ultrapowers are purely infinite.  We show quickly how this gives that purely infinite $C^*$-algebras have purely infinite ultrapowers.

In Section~\ref{sec:eg_pi_ba}, we provide natural examples of Banach algebras which do have purely infinite ultrapowers.  These are built as algebras of operators on suitable Banach spaces. Finally, we show that if a Banach algebra $\mc A$ does have simple ultrapowers, then it behaves a little like a $C^*$-algebra, in the sense that non-zero continuous algebra homomorphisms out of $\mc A$ must be bounded below. We use this property to show that our main example does not have simple ultrapowers. In particular, it cannot have purely infinite ultrapowers either (\textit{cf.\ }Lemma~\ref{l: pui} (1)).

In Section~\ref{sec:cuntz_alg} we present our main construction.  As in \cite{dh}, we use the \emph{Cuntz semigroup} $\text{Cu}_2$, which is a semigroup with zero element, modelled on the relations of the Cuntz algebra $\mc O_2$.  We study the semigroup algebra $\mc A=\ell^1(\text{Cu}_2\setminus\{\lozenge\}, \#)$, where we replace the semigroup zero by the algebra zero.  We recall some of the combinatorics of this semigroup.  There are two natural idempotents in this algebra, and we quotient by the relation that these idempotents sum to $1$, say leading to the algebra $\mc A/\mc J$.  By a delicate combinatorial argument, we show that the resulting Banach algebra is purely infinite: for any non-zero $a\in\mc A/\mc J$ we find $f\in \mc A$ which maps to $a$, and $g,h\in\mc A$ with $g\# f\# h=1$, see Theorem~\ref{t: main_new}.  To show that $\mc A/\mc J$ does not have simple ultrapowers, we construct a faithful, but not bounded below, representation on the Banach space $\ell^p$, for each $p \in [1, \infty)$.

The Banach algebra $\mc A$ has been previously studied in \cite{dlr}, but in relation to being properly infinite (and further we studied a ``weighted'' version of this algebra in \cite{dh}).  The underlying algebra, given by generators and relations, but without the $\ell^1$-norm completion, has a much longer history, as noticed by Phillips in \cite{p1}; compare Remark~\ref{rem:cohn_algs} below.  Indeed, Phillips makes a careful study of (in particular) the algebra $\mc O^p_2$, which, in our language, is the \emph{closure} of the image of $\mc A/\mc J$ in $\mc B(\ell^p)$. It is worth noting here that $\mc A/\mc J$ itself is the $\ell^1$-completion of the Leavitt algebra $L_2$; see Remark~\ref{rem:leavitt_algs}. As we consider in Remark~\ref{rem:phillips}, given the lack of ``permanence'' properties for purely infinite Banach algebras, there appears to be no logical connections between our results and those of Phillips.  In particular, Phillips shows that $\mc O_2^p$ is purely infinite, but we have been unable to decide if $\mc O_2^p$ has purely infinite ultrapowers, or not.

However, an immediate corollary of the material presented in Section~\ref{s: traces} is that $\mc A / \mc J$ and $\mc O^p_2$ are \textit{not} isomorphic for any $p \in [1, \infty)$ (see Theorem~\ref{t: aj_o_2_nonisomorphic}). Section~\ref{s: traces} is devoted to the study of traces on $\mc A / \mc J$ and $\mc O^p_2$. Namely, we show that there \textit{are} non-zero bounded traces on $\mc A / \mc J$ (see Theorem~\ref{t: trace_on_aj}) whilst there are \textit{no} non-zero bounded traces on $\mc O^p_2$, for any $p \in [1, \infty)$ (see Theorem~\ref{t: no_trace_on_o_2}).

Unless stated otherwise, we will use the same notation and terminology as in \cite{dh}.

\subsection{Purely infinite algebras}\label{sec:pi}

Let $\mc A$ be an algebra. We say that two idempotents $p,q \in \mc A$ are \textit{algebraically Murray--von Neumann equivalent} or simply \textit{equivalent} (in notation, $p \sim q$) if there exist $a,b \in \mc A$ such that $p = ab$ and $q=ba$. Note that $\sim$ is an equivalence relation on the set of idempotents of $\mc A$. We say that the idempotents $p,q \in \mc A$ are \textit{orthogonal} (in notation, $p \perp q$) if $pq=0=qp$.  
An idempotent $p$ is \emph{infinite} if $p = q+r$ for orthogonal idempotents $q,r \in \mc A$ with $p\sim q$ and $r\not=0$.

If $\mc A$ is additionally a $*$-algebra, then a self-adjoint idempotent is called a \textit{projection}.  Here one often takes a different notion of equivalence, which for $C^*$-algebras is well-known to give the same definitions; compare \cite[Section~2]{dh}.

The notion of a $C^*$-algebra being \emph{purely infinite} is well-known, and has many equivalent definitions, mostly studied for simple algebras, but also in the non-simple case, \cite{kr}.  Purely infinite $C^*$-algebras appear prominently in the classification programme for $C^*$-algebras, \cite{p}, in particular in the guise of the \emph{Kirchberg algebras}.  It is common to take as a definition that a $C^*$-algebra is purely infinite if every hereditary subalgebra contains an infinite projection.

In a more general direction, the notion of a simple ring being purely infinite was studied in \cite{agp}, where it is taken as definition that a simple ring $R$ is purely infinite if every right (or equivalently, left) ideal of $R$ contains an infinite idempotent.  Consideration of what it means for a non-simple ring to be purely infinite is given in \cite{agps}.

Common to both definitions (in the simple case) is the following equivalence; for $C^*$-algebras see for example \cite[Theorem~V.5.5]{dav} while for rings see \cite[Theorem~1.6]{agp}.

\begin{definition}
A complex unital algebra $\mc A$ is \emph{purely infinite} if it is not a division algebra and for every $a \in \mc A$ non-zero there exist $b,c \in \mc A$ such that $1_{\mc A} = bac$.
\end{definition}

In this paper, we shall work only with this definition.
Note that by the Gel'fand--Mazur Theorem a complex unital normed algebra is a division algebra if and only if it is isomorphic to the field of complex numbers $\mathbb{C}$. In the rest of the paper all algebras are assumed to be complex.

We finish the section with the following.
We recall that a unital algebra $\mc A$ is \textit{properly infinite} if there exist idempotents $p,q \in \mc A$ with $p \sim 1_{\mc A}$, $q \sim 1_{\mc A}$ and $p \perp q$.

\begin{lemma}\label{l: pui}
	Let $\mc A$ be a purely infinite algebra. Then $\mc A$ is
	\begin{enumerate}
		\item[(1)] simple; and
		\item[(2)] properly infinite.
	\end{enumerate}
\end{lemma}	

\begin{proof}
	We first show that $\mc A$ is simple. Let $\mc J$ be a non-zero, two-sided ideal in $\mc A$ and pick $a \in \mc J$ non-zero. There exist $b,c \in \mc A$ such that $1_{\mc A} = bac$, hence $1_{\mc A} \in \mc J$. Thus $\mc J = \mc A$.
	
	We now show that $\mc A$ is properly infinite. Recall that $\mc A$ is not a division algebra, hence we can find a non-zero, non-invertible element, say  $a \in \mc A$.  Let $b,c \in \mc A$ be such that $1_{\mc A} = bac$. We define $p:=cba$ and $r:=acb$, it is clear that $p,r \in \mc A$ are idempotents with $p \sim 1_{\mc A} \sim r$. However $p$ and $r$ need not be orthogonal. Nevertheless, either $p \neq 1_{\mc A}$ or $r \neq 1_{\mc A}$ (or both), otherwise $a$ were invertible with inverse $cb$ which is not possible. Without loss of generality we may assume $p \neq 1_{\mc A}$. Let $s:= 1_{\mc A}-p$, then $s \in \mc A$ is a non-zero idempotent with $s \perp p$. We can find some $x,y \in \mc A$ such that $1_{\mc A} = xsy$. Define $q:= syxs$, then $q^2 = syxssyxs=syxsyxs=syxs=q$, and clearly $q \sim 1_{\mc A}$. Now $p \perp q$ as $p \perp s$.
\end{proof}

\subsection{Ultrapowers of Banach algebras}

Let $\mc A$ be a Banach algebra and let $\ell^\infty(\mc A)$ be the Banach space of all bounded sequences $(a_n)$ in $\mc A$, turned into a Banach algebra with pointwise operations.  Let $\mc U$ be a non-principal ultrafilter on $\mathbb N$ and let $c_{\mc U}(\mc A)$ be the closed, two-sided ideal of $\ell^\infty(\mc A)$ formed of sequences $(a_n)$ with $\lim_{n\rightarrow\mc U} \|a_n\|=0$. The quotient
\begin{align}
(\mc A)_{\mc U} = \ell^\infty(\mc A) / c_{\mc U}(\mc A)
\end{align}
is the \emph{ultrapower}, see \cite{hein}. It is well known that $(\mathbb C)_{\mc U} \cong \mathbb{C}$.

We shall denote by a capital letter $A$, and so forth, an element $A=(a_n) \in \ell^\infty(\mc A)$.  Let $\pi_{\mc U, \mc A} \colon \ell^\infty(\mc A) \rightarrow(\mc A)_{\mc U}$ be the quotient map; then
\begin{align}
\|\pi_{\mc U, \mc A}(A)\| = \lim_{n\rightarrow\mc U} \|a_n\|.
\end{align}
In particular, given any $a\in (\mc A)_{\mc U}$ we can always find $A=(a_n)\in\ell^\infty(\mc A)$ with $\pi_{\mc U, \mc A}(A)=a$ and $\|A\|=\sup_n \|a_n\| = \|a\|$. We always assume that our ultrafilters are non-principal, which on a countable indexing set, is equivalent to being \emph{countably-incomplete} (see \cite[Section~1]{hein}). When it does not cause any confusion we may drop the subscripts on $\pi_{\mc U, \mc A}$.
	
Given a Banach algebra $\mc A$ and an ultrafilter $\mc U$, the ``diagonal'' map
\begin{align}
\iota_{\mathcal{A}} \colon \mathcal{A} \rightarrow (\mathcal{A})_{\mathcal{U}}; \quad a \mapsto \pi_{\mathcal{U}}((a))
\end{align}
is an isometric embedding. In particular, $\dim \mc A \leqslant \dim (\mathcal{A})_{\mathcal{U}}$ follows. \smallskip	

Let $\mc A$ and $\mc B$ be Banach algebras and let $\psi \colon \mc A \to \mc B$ be a continuous algebra homomorphism. Then for an ultrafilter $\mc U$ there is an induced continuous algebra homomorphism $\psi_{\mc U} \colon (\mc A)_{\mc U} \to (\mc B)_{\mc U}$ such that 
\begin{align}\label{e: ultra_induced_hom}
\psi_{\mc U} \left( \pi_{\mc U, \mc A} \left( (a_n) \right) \right) = \pi_{\mc U, \mc B} \left( \left( \psi(a_n) \right) \right)
\end{align}
for each $(a_n) \in \ell^\infty(\mc A)$. If $\psi$ is non-zero then $\psi_{\mc U}$ is non-zero too.

\section{Norm control}\label{sec:norm_control}

\subsection{Quantifying pure infiniteness}
In \cite{dh} we ``quantified'' Dedekind-finiteness, proper infiniteness and stable rank one, in order to characterise when an ultrapower $(\mc A)_{\mc U}$ has these ring-theoretic properties of the underlying Banach algebra $\mc A$. We follow our previous approach in the present paper.

\begin{definition}
	Let $\mc A$ be a unital Banach algebra.  For $a\in \mc A \setminus \{0\}$ define
	\[ C_{\text{pi}}^{\mc A}(a) = \inf \{ \|b\| \|c\| \colon b,c \in \mc A, \, bac=1 \} \]
	with $C_{\text{pi}}^{\mc A}(a) = \infty$ if there are no $b,c\in \mc A$ with $bac=1$.
\end{definition}

Then a unital Banach algebra $\mc A$ is purely infinite exactly when it is not isomorphic to $\mathbb{C}$ and $C_{\text{pi}}^{\mc A}(a) < \infty$ for each $a \in \mc A \setminus \{0 \}$. Note that if $a \in \mc A$ is such that $C_{\text{pi}}^{\mc A}(a) < \infty$ then $1/ \|a\| \leqslant C_{\text{pi}}^{\mc A}(a)$.

By homogeneity, we have
\begin{align}
C_{\text{pi}}^{\mc A}(za) = |z|^{-1} C_{\text{pi}}^{\mc A}(a) \qquad (a \in \mc A \setminus \{0\}, \; z \in \mathbb{C} \setminus \{0\}).
\end{align}
Thus it is enough to study the unit sphere of $\mc A$.

Usually, we will drop the superscript on $C_{\text{pi}}^{\mc A}(a)$ and simply write $C_{\text{pi}}(a)$, whenever it is clear from the context which Banach algebra the element $a$ is taken from.

\begin{proposition}\label{p: ultrapurelyinf}
	Let $\mc U$ be a countably-incomplete ultrafilter. Let $\mc A$ be a unital Banach algebra not isomorphic to $\mathbb{C}$. Then the following are equivalent.
	\begin{enumerate}
		\item[(1)] $(\mc A)_{\mc U}$ is purely infinite;
		\item[(2)] There is $K>0$ such that $C_{\text{pi}}(a) < K$ for each $a \in \mc A$ with $\|a\|=1$.
	\end{enumerate}
\end{proposition}

\begin{proof}
	As usual, we may suppose that $\mc U$ is a non-principal ultrafilter on $\mathbb N$. 
	
	$((1) \Rightarrow (2)):$ We prove the statement by way of a contraposition. Assume $(2)$ does not hold. Then in particular we can pick a sequence $(a_n)$ in $\mc A$ consisting of norm one elements such that $C_{\text{pi}}(a_n) > n$ for each $n \in \mathbb{N}$. Let $A:= (a_n)$ so $A \in \ell^{\infty}(\mc A)$. Assume towards a contradiction that $(\mc A)_{\mc U}$ is purely infinite. Thus we can find $B=(b_n)$, $C=(c_n) \in \ell^{\infty}(\mc A)$ such that $\pi(1) = \pi(B) \pi(A) \pi(C)$, or equivalently, $\lim_{n \rightarrow \mc U} \| 1-b_n a_n c_n \| =0$. Let $\mc N := \{ n \in \mathbb{N} \colon \|1 - b_n a_n c_n \| < 1/2 \}$, then $\mc N \in \mc U$. By the Carl Neumann series $x_n :=b_n a_n c_n \in \inv(\mc A)$ with $\|x_n^{-1} \| \leqslant 2$ for each $n \in \mc N$. As $1=x_n^{-1}x_n =(x_n^{-1}b_n)a_n c_n$, we conclude that
	\begin{align}
	n< C_{\text{pi}}(a_n) \leqslant \|x_n^{-1} b_n \| \|c_n \| \leqslant \|x_n^{-1} \| \| b_n \| \|c_n \| \leqslant 2 \|B \| \|C \| \qquad (n \in \mc N). 
	\end{align}
	As $\mc N \in \mc U$ and thus $\mc N$ is infinite, this gives a contradiction.
	
	$((2) \Rightarrow (1)):$ Assume $(2)$ holds. Let $A=(a_n) \in \ell^{\infty}(\mc A)$ be such that $\pi(A) \neq 0$. This is equivalent to saying that $\lim_{n \rightarrow \mc U} \| a_n \| \neq 0$, hence there is $\delta >0$ such that $\{ n \in \mathbb{N} \colon \| a_n \| < \delta \} \notin \mc U$, that is, $\mc M := \{ n \in \mathbb{N} \colon \|a_n \| \geqslant \delta \} \in \mc U$. Thus we may set $a'_n := a_n / \|a_n \|$ whenever $n \in \mathcal{M}$, and $a'_n := 0$ otherwise. Clearly $\|a'_n\| =1$ for each $n \in \mathcal{M}$, hence by the assumption it follows that $C_{\text{pi}}(a'_n) < K$ for each $n \in \mc M$. Thus for every $n \in \mc M$ we can find $b'_n, c'_n \in \mc A$ such that $b'_n a'_n c'_n =1$ and $\|b'_n \| \|c'_n \| < K$. We set 
	\begin{align}
	b_n :=\begin{cases}
	\sqrt{\dfrac{\|c'_n \|}{\|b'_n \| \|a_n \|}} b'_n &\text{ if } n \in \mc M, \\
	0 &\text{ otherwise;}
	\end{cases} \quad \text{ and } \quad 
	c_n :=\begin{cases}
	\sqrt{\dfrac{\|b'_n \|}{\|c'_n \| \|a_n \|}} c'_n &\text{ if } n \in \mc M, \\
	0 &\text{ otherwise.}
	\end{cases}
	\end{align}
	Hence $b_n a_n c_n = \|a_n \|^{-1} b'_n a_n c'_n = b'_n a'_n c'_n =1$ for each $n \in \mc M$. It is also follows from the definitions that $\|b_n\| =\sqrt{\|b'_n \| \|c'_n \| / \|a_n\|} < \sqrt{K / \delta}$ and similarly $\|c_n\|  < \sqrt{K / \delta}$, hence $B:=(b_n)$, $C:=(c_n) \in \ell^{\infty}(\mc A)$. 
	
	Fix $\varepsilon >0$. Then
	\begin{align}
	\mc M = \{ n \in \mathbb{N} \colon 1=b_n a_n c_n \} \subseteq \{ n \in \mathbb{N} \colon \|1-b_n a_n c_n \| < \varepsilon \},
	\end{align}
	hence from $\mc M \in \mc U$ we conclude $\{ n \in \mathbb{N} \colon \|1-b_n a_n c_n \| < \varepsilon \} \in \mc U$. Thus $\lim_{n \rightarrow \mc U} \| 1- b_n a_n c_n \| =0$, which is equivalent to $\pi(B) \pi(A) \pi(C) = \pi(A)$. Lastly, $(\mc A)_{\mc U}$ cannot be isomorphic to $\mathbb{C}$ by the assumption and $\dim \mc A \leqslant \dim (\mathcal{A})_{\mathcal{U}}$. Thus $(\mc A)_{\mc U}$ is purely infinite.
\end{proof}

\begin{remark}
	In view of the comment before Proposition~\ref{p: ultrapurelyinf}, we may rewrite condition $(2)$ as
	\begin{enumerate}
		\item[(2')] There is $K>0$ such that $C_{\text{pi}}(a) \leqslant K / \|a\|$ for each non-zero $a \in \mc A$. 
	\end{enumerate}
	Consequently $(\mc A)_{\mc U}$ is purely infinite if and only if there exists a $K>0$ such that 
	\begin{align}
	1/ \|a\| \leqslant C_{\text{pi}}(a) \leqslant K / \|a\| \qquad (a \in \mc A \setminus \{0\}).
	\end{align}
\end{remark}

\begin{corollary}\label{c: ultra_pi_implies}
	Let $\mc U$ be a countably-incomplete ultrafilter, and let $\mc A$ be a Banach algebra such that $(\mc A)_{\mc U}$ is purely infinite. Then $\mc A$ is purely infinite.
\end{corollary}

\begin{proof}
	Note that $\mc A$ is not isomorphic to $\mathbb{C}$, otherwise $(\mc A)_{\mc U} \cong  (\mathbb{C})_{\mc U} \cong \mathbb{C}$ which is nonsense. Hence by the assumption we can take some $K>0$ which satisfies the conditions of Proposition~\ref{p: ultrapurelyinf} (2). Let $a \in \mc A$ be non-zero. We set $a':= a / \|a\|$, then $C_{\text{pi}}(a') <K$. Thus there exist $b',c' \in \mc A$ such that $b'a'c'=1$. Now define $b:= b' / \sqrt{\|a \|}$ and $c:= c' / \sqrt{\|a\|}$, thus $bac=1$ as required.
\end{proof}

In fact, we can make a quantitative statement in this direction.
\begin{lemma}
	Let $\mathcal{A}$ be a unital Banach algebra. If $\mc U$ is an ultrafilter, then
	\begin{align}
	C^{(\mathcal{A})_{\mathcal{U}}}_{\text{pi}} \circ \iota_{\mathcal{A}} = C^{\mathcal{A}}_{\text{pi}}.
	\end{align}
\end{lemma}

\begin{proof}
	\textit{$(\geqslant)$}: Let $a \in \mathcal{A}$, and put $A :=(a) \in \ell^{\infty}(\mathcal{A})$. Assume $B=(b_n)$, $C=(c_n) \in \ell^{\infty}(\mathcal{A})$ are such that $\pi_{\mathcal{U}}(1) = \pi_{\mathcal{U}}(B) \pi_{\mathcal{U}}(A) \pi_{\mathcal{U}}(C)$, which is equivalent to $\lim_{n \rightarrow \mathcal{U}} \| 1- b_nac_n \| = 0$. Let us fix $\varepsilon \in (0,1)$. Then
	\begin{align}
	\mathcal{N}_{\varepsilon} := \{ n \in \mathbb{N} \colon \| 1- b_nac_n \| < \varepsilon \} \in \mathcal{U},
	\end{align}
	and by the Carl Neumann series $x_n:=b_n a c_n \in \inv(\mathcal{A})$ with $\|x_n^{-1}\| \leqslant (1-\varepsilon)^{-1}$ for each $n \in \mathcal{N}_{\varepsilon}$. Thus $1 = x_n^{-1}x_n = (x_n^{-1}b_n)ac_n$, and consequently
	\begin{align}
	C^{\mathcal{A}}_{\text{pi}}(a) \leqslant \| x_n^{-1}b_n \| \|c_n \| \leqslant \| x_n^{-1} \| \|b_n \| \|c_n \| < (1-\varepsilon)^{-1} \|b_n \| \|c_n \| \quad (n \in \mathcal{N}_{\varepsilon}).
	\end{align}
	Therefore $C^{\mathcal{A}}_{\text{pi}}(a) \leqslant \lim_{n \rightarrow \mathcal{U}} \|b_n \| \|c_n \| (1- \varepsilon)^{-1} = \| \pi_{\mathcal{U}}(B) \| \| \pi_{\mathcal{U}}(C) \| (1- \varepsilon)^{-1}$, which holds for all $\varepsilon \in (0,1)$, hence $C^{\mathcal{A}}_{\text{pi}}(a) \leqslant \| \pi_{\mathcal{U}}(B) \| \| \pi_{\mathcal{U}}(C) \|$. Consequently $C^{\mathcal{A}}_{\text{pi}}(a) \leqslant C^{(\mathcal{A})_{\mathcal{U}}}_{\text{pi}}(\pi_{\mathcal{U}}(A))= C^{(\mathcal{A})_{\mathcal{U}}}_{\text{pi}}(\iota_{\mathcal{A}}(a))$, as claimed. \smallskip
	
	\textit{$(\leqslant)$}: Let $a \in \mathcal{A}$. Assume $b,c \in \mathcal{A}$ are such that $1=bac$. Putting $A:=(a), B:=(b), C:=(c) \in \ell^{\infty}(\mathcal{A})$, we clearly have $\pi_{\mathcal{U}}(1) =\pi_{\mathcal{U}}(B) \pi_{\mathcal{U}}(A) \pi_{\mathcal{U}}(C)$. Consequently
	\begin{align}
	C^{(\mathcal{A})_{\mathcal{U}}}_{\text{pi}}(\iota_{\mathcal{A}}(a)) = C^{(\mathcal{A})_{\mathcal{U}}}_{\text{pi}}(\pi_{\mathcal{U}}(A)) \leqslant \| \pi_{\mathcal{U}}(B) \| \| \pi_{\mathcal{U}}(C) \| = \|b \| \|c \|,
	\end{align}
	and therefore $C^{(\mathcal{A})_{\mathcal{U}}}_{\text{pi}}(\iota_{\mathcal{A}}(a)) \leqslant C^{\mathcal{A}}_{\text{pi}}(a)$, as required.
\end{proof}

One might wonder whether the converse to Corollary~\ref{c: ultra_pi_implies} could be true. We will show that this is not the case: there is a purely infinite Banach $*$-algebra which does not have purely infinite ultrapowers (see Theorems~\ref{t: quot_pi} and~\ref{t: ultra_not_pi} ). \smallskip

However, it is well known (see \cite[Section~3.13.7]{fhlrtvw}) that the converse to Corollary~\ref{c: ultra_pi_implies} remains true for $C^*$-algebras. Here we demonstrate how this can easily be deduced from Proposition~\ref{p: ultrapurelyinf}.

\begin{lemma}
	Let $\mc A$ be a purely infinite $C^*$-algebra. Then $C_{\text{pi}}(a) =1$ for each $a \in \mc A$ with $\|a\|=1$, consequently $(\mc A)_{\mc U}$ is purely infinite for every countably-incomplete ultrafilter $\mc U$.
\end{lemma}

\begin{proof}
	Let $a \in \mc A$ have norm one. Let us fix $\varepsilon >0$. Clearly $a^*a \in \mc A$ is positive, hence by \cite[Theorem~V.5.5]{dav} there is some $x \in \mc A$ such that $(xa^*)ax^*=x(a^*a)x^* =1$ and $\|x\|< \|a^*a\|^{-1/2} + \varepsilon = 1 + \varepsilon$. Thus
	\begin{align}
	C_{\text{pi}}(a) \leqslant \|xa^*\| \|x^*\| \leqslant \|x\|^2 \|a\| <(1 + \varepsilon)^2,
	\end{align}
	and therefore $C_{\text{pi}}(a) \leqslant 1$. The ``consequently'' part follows from Proposition~\ref{p: ultrapurelyinf}.
\end{proof}

We note that \cite[Theorem~V.5.5]{dav} has an elementary (functional calculus) proof, passing by way of an equivalent definition of what purely infinite means for $C^*$-algebras, compare our discussion in Section~\ref{sec:pi}.
\medskip

We briefly consider the asymptotic sequence algebra. Let $c_0(\mc A)$ be the closed, two-sided ideal of $\ell^\infty(\mc A)$ which consists of sequences $(a_n)$ with $\lim_{n \to \infty} \|a_n\|=0$.  In fact, when $\mc A$ is unital, $\ell^\infty(\mc A)$ is the \emph{multiplier algebra} of $c_0(\mc A)$ (compare \cite[Section~13]{far} for example). The \emph{asymptotic sequence algebra}
$\asy(\mc A)$ is the quotient algebra $\ell^\infty(\mc A) / c_0(\mc A)$.

As opposed to the previously studied properties in \cite{dh} such as stable rank one, Dedekind-finitess and proper infiniteness, the theory for the asymptotic sequence algebra and the ultrapower of a Banach algebra seems to bifurcate here.

\begin{remark}\label{r: asy_not_pi}
	Let $\mc A$ be a non-zero unital Banach algebra. Then $\asy(\mc A)$ is not simple and hence not purely infinite.
\end{remark}

\begin{proof}
	Note that $\asy (\mc A)$ is simple if and only if $c_0(\mc A)$ is a maximal two-sided ideal in $\ell^{\infty}(\mc A)$. But this latter is not possible, as for example, the following shows. Let $\mc U$ be a non-principal ultrafilter on $\mathbb{N}$ such that $2 \mathbb{N} \in \mc U$. Let $A:=(a_n)$ be a sequence in $\mc A$ defined by $a_{2n}:=1_{\mc A}$ and $a_{2n-1}:= 0_{\mc A}$ for each $n \in \mathbb N$. Clearly $A \in \ell^{\infty}(\mc A)$ and in fact $A \in c_{\mc U}(\mc A)$ by definition. On the other hand clearly $A \notin c_0(\mc A)$. Consequently $c_0(\mc A) \subsetneq c_{\mc U}(\mc A)$ which shows that $c_0(\mc A)$ cannot be maximal. The last part follows from Lemma~\ref{l: pui}.
\end{proof}

We finish this section with a handy tool for showing when Banach algebras fail to have simple ultrapowers.  Indeed, this is one of the key ideas in the proof of Theorem~\ref{t: ultra_not_pi}.

	\begin{proposition}\label{p: pi_implies_bdd_below}
		Let $\mc A$ be a Banach algebra such that $(\mc A)_{\mc U}$ is simple for some countably-incomplete ultrafilter $\mc U$. Then for every Banach algebra $\mc B$, every non-zero continuous algebra homomorphism $\psi \colon \mc A \to \mc B$ is bounded below.
	\end{proposition}
	
	\begin{proof}
		We prove by contraposition. Suppose $\mc B$ is a Banach algebra and $\psi \colon \mc A \to \mc B$ is a non-zero continuous algebra homomorphism which is not bounded below. Thus we can pick a sequence $(a_n)$ in $\mc A$ consisting of norm one elements with $\lim_{n \to \infty}\| \psi(a_n)\| =0$. In particular $A:= (a_n)\in \ell^{\infty}(\mc A)$ and $(\psi(a_n)) \in c_{\mathcal{U}}(\mathcal{B})$, where $\mathcal{U}$ is a countably-incomplete ultrafilter. Consider the induced continuous algebra homomorphism $\psi_{\mc U} \colon (\mc A)_{\mc U} \to (\mc B)_{\mc U}$. On the one hand, from \eqref{e: ultra_induced_hom} and the above we see that $\psi_{\mc U}(\pi_{\mc U, \mc A}(A)) = \pi_{\mc U, \mc B}((\psi(a_n)))=0$, hence $\pi_{\mc U, \mc A}(A) \in \Ker(\psi_{\mc U})$. As $\| \pi_{\mc U, \mc A}(A) \| = \lim_{n \to \mc U}\| a_n \| =1$, it follows that $\Ker(\psi_{\mc U})$ is a non-zero ideal in $(\mc A)_{\mc U}$. On the other hand, $\psi$ and hence $\psi_{\mc U}$ is non-zero, therefore $\Ker(\psi_{\mc U})$ is also a proper ideal in $(\mc A)_{\mc U}$. So $(\mc A)_{\mc U}$ is not simple.
	\end{proof}

\subsection{Examples of Banach algebras with purely infinite ultrapowers}\label{sec:eg_pi_ba}

It is a good point to give an example of a class of non-$C^*$, Banach algebras with purely infinite ultrapowers. In what follows, $\mc B (X)$ and $\mc K (X)$ denote the algebra of bounded linear operators on a Banach space $X$ and the set of compact operators on $X$, respectively. Clearly $\mc B (X)$ is a unital Banach algebra and $\mc K (X)$ is a closed, two-sided ideal in $\mc B (X)$.

\begin{proposition}\label{p: lpc0}
	Let $X$ be $c_0$ or $\ell^p$, where $1 \leqslant p < \infty$. Then $(\mathcal{B}(X)/ \mc K (X))_{\mathcal{U}}$ is purely infinite if $\mc U$ is a countably-incomplete ultrafilter. More precisely, for every $a \in \mathcal{B}(X)/ \mc K (X)$ with $\|a\|=1$ there exist $b,c \in \mathcal{B}(X)/ \mc K (X)$ such that $1=bac$ and $\|b\| \|c\|=1$.
\end{proposition}

The proof of Proposition \ref{p: lpc0} relies on the following result of Ware, see \cite[Lemma~3.3.6]{ware}. Note that it is a strict strengthening of \cite[Lemma~2.1]{tylli}; the proof works by extracting a suitable
block basic sequence equivalent to the standard unit vector basis for $X$.

\begin{lemma}\label{l: ware}
	Let $X$ be $c_0$ or $\ell^p$, where $1 \leqslant p < \infty$. Then for each $A \in \mc B (X)$ a non-compact operator, there exist $B,C \in \mc B (X)$ such that
	\begin{align}
	I_X = BAC \quad \text{and} \quad \| \pi(B)\| \|\pi(C) \| = 1/ \|\pi(A)\|,
	\end{align}
	where $\pi \colon \mc B (X) \to \mc B (X) / \mc K (X)$ the quotient map.
\end{lemma}

\begin{proof}[Proof of Proposition~\ref{p: lpc0}]
	Let $A \in \mc B (X)$ be such that $\| \pi (A) \| =1$. Hence by Lemma~\ref{l: ware} there are $B,C \in \mc B (X)$ such that $I_X = BAC$ and $\| \pi(B)\| \|\pi(C) \| = 1$. This obviously proves the first part of the claim. In particular, $C_{\text{pi}}(\pi(A)) = 1$ follows whenever $\| \pi (A) \| =1$. Now Proposition~\ref{p: ultrapurelyinf} yields that $(\mc B (X) / \mc K (X))_{\mc U}$ is purely infinite, whenever $\mc U$ is a countably-incomplete ultrafilter.
\end{proof}
\medskip 

Let us introduce some terminology, commonly found in the literature, for the property we have been studying.
For a unital algebra $\mc A$, and given $a \in \mc A$, we say that \emph{$1_{\mc A}$ factors through $a$}, or \emph{$a$ is a purely infinite element}, if there exist $b,c \in \mc A$ such that $1_{\mc A} = bac$.

In a unital algebra $\mc A$ we define the set
\begin{align}
\mc M_{\mc A}:= \lbrace a \in \mc A \colon 1_{\mc A} \text{ does not factor through } a \rbrace.
\end{align}
The following result is folklore and easy to see; we omit the proof.

\begin{proposition}\label{p: uniquemax}
	Let $\mc A$ be a unital algebra.
	\begin{itemize}
		\item The set $\mc M_{\mc A}$ is closed under scalar multiplication, and under multiplying elements of it from the left and right by elements from $\mc A$. Thus it is the largest proper (and therefore unique maximal) two-sided ideal in $\mc A$ if and only if $\mc M_{\mc A}$ is closed under addition.
		\item If $\mc M_{\mc A}$ is closed under addition and $\mc A / \mc M_{\mc A}$ is not a division algebra, then $\mc A / \mc M_{\mc A}$ is purely infinite.
	\end{itemize}
\end{proposition}

Note that in the second item the condition that $\mc A / \mc M_{\mc A}$ is not a division algebra cannot be omitted. Indeed, Kania and Laustsen showed in \cite[Theorem~1.2]{kl1} that with $X:= C[0, \omega_1]$, the one-codimensional \textit{Loy--Willis ideal} coincides with $\mc M_{\mc B(X)}$ and hence $\mc B (X) / \mc M_{\mc B(X)} \cong \mathbb{C}$. \smallskip

When the unital Banach algebra $\mc A$ is $\mc B (X)$ for some ``classical'' Banach space $X$, it happens very often that $\mc M_{\mc B (X)}$ is the unique maximal ideal in $\mc M_{\mc B (X)}$. Here we give a few examples, a more comprehensive list can be found in \cite[p.~4832]{kl1}.

\begin{example}\label{classical}
	If $X$ is any of the Banach spaces below then $\mc M_{\mc B(X)}$ is closed under addition and hence it is the unique maximal ideal in $\mc B (X)$:
	\begin{itemize}
		\item $X= c_0$ or $X= \ell^p$, where $1 \leqslant p < \infty$, in this case $\mc M_{\mc B(X)} = \mc K (X)$ (see \cite{gmf});
		\item $X=\ell^{\infty}$ (see \cite[p.~253]{ll});
		\item $X= L^p[0,1]$, where $1 \leqslant p <\infty$ (see \cite[Theorem~1.3 and the text after]{djsch});
		\item $X= C[0,1]$ (see the explanation in \cite[p.~4832]{kl1}).
	\end{itemize}
\end{example}

\begin{remark}
	Let $\mc A$ be a unital algebra and let $\mc J$ be a two-sided ideal in $\mc A$ such that $\mc A / \mc J$ is purely infinite. Then $\mc M_{\mc A}$ is closed under addition if and only if $\mc J = \mc M_{\mc A}$. Indeed, $\mc A / \mc J$ is simple by Lemma~\ref{l: pui}, or equivalently, $\mc J$ is a maximal ideal. Hence if $\mc M_{\mc A}$ is closed under addition then it is the unique maximal ideal in $\mc A$ by Proposition~\ref{p: uniquemax}, thus $\mc J = \mc M_{\mc A}$. The other direction is trivial.
\end{remark}

It is certainly not true however that for a unital Banach algebra $\mc A$ and a closed, two-sided ideal $\mc J$ of $\mc A$ the quotient $\mc A / \mc J$ is purely infinite only if $\mc M_{\mc A}$ is closed under addition. We shall show this by way of a counter-example. In order to do this, let us recall the following piece of terminology. For Banach spaces $X$ and $Y$ the symbol $\overline{\mathcal{G}}_{Y}(X)$ denotes the closed, two-sided ideal of operators on $X$ which approximately factor through $Y$.

\begin{lemma}\label{l: lpsums}
	Let $X:=\ell^p \oplus \ell^q$, where $1 \leqslant p < q < \infty$. Then $M_{\mathcal{B}(X)}$ is not closed under addition while $\left(\mathcal{B}(X)/ \overline{\mathcal{G}}_{Y}(X) \right)_{\mathcal{U}}$ is purely infinite, where $Y$ is $\ell^p$ or $\ell^q$ and $\mc U$ is a countably-incomplete ultrafilter. More precisely, for every $a \in \mathcal{B}(X)/ \overline{\mathcal{G}}_{Y}(X)$ with $\|a\|=1$ there exist $b,c \in \mathcal{B}(X)/ \overline{\mathcal{G}}_{Y}(X)$ such that $1=bac$ and $\|b\| \|c\|=1$.
\end{lemma}

\begin{proof}
	The first part of the claim is well known; see \textit{e.g.}\ \cite[Theorem~5.3.2]{pie}. Indeed, $\mc B(X)$ has exactly two maximal two-sided ideals, namely, $\overline{\mathcal{G}}_{\ell^p}(X)$ and $\overline{\mathcal{G}}_{\ell^q}(X)$. We will work with $Y= \ell^p$, the other case is entirely analogous.
	
	Let us recall that by Pitt's Theorem \cite[Theorem~2.1.4]{ak}, we can describe $\mathcal{B}(X)$ and $\overline{\mathcal{G}}_{\ell^p}(X)$ as
	\begin{align*}
	\mc B (X) &=
	\begin{bmatrix}
	\mc B (\ell^p) & \mc B(\ell^q, \ell^p) \\
	\mc B(\ell^p, \ell^q) & \mc B(\ell^q)
	\end{bmatrix} =
	\begin{bmatrix}
	\mc B (\ell^p) & \mc K(\ell^q, \ell^p) \\
	\mc B(\ell^p, \ell^q) & \mc B(\ell^q)
	\end{bmatrix}, \\
	\overline{\mathcal{G}}_{\ell^p}(X) &=
	\begin{bmatrix}
	\mc K (\ell^p) & \mc B(\ell^q, \ell^p) \\
	\mc B(\ell^p, \ell^q) & \mc B(\ell^q)
	\end{bmatrix} =
	\begin{bmatrix}
	\mc K (\ell^p) & \mc K(\ell^q, \ell^p) \\
	\mc B(\ell^p, \ell^q) & \mc B(\ell^q)
	\end{bmatrix}. 
	\end{align*}
	Consequently,
	\begin{align*}
	\mc B (X) / \overline{\mathcal{G}}_{\ell^p}(X)  \cong
	\mc B (\ell^p) / \mc K (\ell^p),
	\end{align*}
	where the isomorphism is clearly isometric. Hence the result follows from Propositions~\ref{p: lpc0} and~\ref{p: ultrapurelyinf}.
\end{proof}

\section{A ``Banach- analogue'' of the Cuntz algebra}\label{sec:cuntz_alg}

In this section we show that a naturally occurring infinite-dimensional Banach $*$-algebra is purely infinite, but it does not have a simple ultrapower.

\subsection{Preliminaries}

\subsubsection{Involutive semigroups with zero elements, and the Banach $*$-algebra $\ell^1(S \setminus \{\lozenge\})$}

We recall that a semigroup $S$ is \textit{involutive} if there is a map $s \mapsto s^*$, $S \to S$ with the property $(s^*)^*=s$ and $(st)^*=t^*s^*$ for each $s,t \in S$. \smallskip

We say that $S$ is a \textit{monoid with a zero element} if $S$ is a monoid with at least two
elements and there exists a $\lozenge \in S$ such that $\lozenge s = \lozenge = s \lozenge$ for
all $s \in S$.  If such a $\lozenge \in S$ exists then it is necessarily unique.  As we assume that
$S$ has more than one element, we have $\lozenge$ is different from the multiplicative
identity $e \in S$. Note that if $S$ is additionally involutive, then necessarily $\lozenge^*= \lozenge$. \smallskip

Let us briefly recall that it is possible to endow the Banach space $\ell^1(S\setminus \{ \lozenge \})$ with a unital Banach algebra structure; see \cite{dlr} and \cite{dh} for details; compare also \cite{gw}.
This is accomplished by identifying $\ell^1(S\setminus \{ \lozenge \})$ with the quotient
algebra $\ell^1(S) / \mathbb{C} \delta_{\lozenge}$, where $\ell^1(S)$ is endowed with the convolution product. This allows us to define a product $\#$ on $\ell^1(S\setminus \lbrace \lozenge \rbrace)$ which satisfies
\begin{align}\label{likeconvbutitsnot}
\delta_s \# \delta_t = \left\{
\begin{array}{l l}
\delta_{st} & \quad \text{if  } st \neq \lozenge \\
0 & \quad \text{if } st = \lozenge \\
\end{array} \qquad \left(s,t \in S\setminus \lbrace \lozenge \rbrace \right). \right.
\end{align}
In particular it follows from equation~\eqref{likeconvbutitsnot} that $(\ell^1(S\setminus \lbrace \lozenge \rbrace), \#)$ is a unital Banach algebra with $\delta_e$ being the unit, and such that $\Vert \delta_e \Vert =1$. 

If in addition $S$ is involutive, then the formula
\begin{align}
f^*(s):= \overline{f(s^*)} \qquad (f \in \ell^1(S \setminus \{ \lozenge\}), \; s \in S \setminus \{ \lozenge\})
\end{align}
defines an isometric involution on $\ell^1(S \setminus \{ \lozenge\})$. Hence $\ell^1(S \setminus \{ \lozenge\})$ is a Banach $*$-algebra.

\subsubsection{The Cuntz semigroup $\text{Cu}_2$}	

In the following $\text{Cu}_2$ denotes the second Cuntz semigroup (see also \cite[Definition~2.2,~p.~141]{renault};
this is also occasionally called the ``polycyclic monoid'' in the literature, \cite{cm}). 
(We warn the reader that ``Cuntz semigroup'' now also means something unrelated in $C^*$-algebra theory.)
That is, $\text{Cu}_2$ is an involutive semigroup with multiplicative identity $e$ and zero element $\lozenge$, and generators $s_1, s_2, s_1^*, s_2^*$ subject to the relations $s_1^*s_1=e=s_2^*s_2$ and $s_1^*s_2 = \lozenge = s_2^* s_1$. In notation, $\text{Cu}_2$ is
\begin{align}
\langle s_1, s_2, s_1^*, s_2^* \colon s_1^* s_1 = e = s_2^* s_2, \; s_1^* s_2 = \lozenge = s_2^* s_1 \rangle.
\end{align}

We now mostly follow the notation of \cite[Section~3.3]{dlr}.

\begin{definition}
	We set
	\begin{align}
	\mathbf{I}_n:= \{(i_1,i_2, \ldots, i_n) \colon i_1, i_2, \ldots, i_n \in \{1,2\} \} \qquad (n \in \mathbb{N}),
	\end{align}
	and $\mathbf{I}_0 := \{\emptyset\}$. Let $\mathbf{I} := \bigcup_{n \in \mathbb{N}_0} \mathbf{I}_n$, and $\mathbf{L}:= \prod_{n \in \mathbb{N}} \{1,2\}$.
	
	Let $\mathbf{n} =(n_i) \in \mathbf{L}$, we then set
	\begin{align}
	\mathbf{n}_0 &:= \emptyset, \notag \\
	\mathbf{n}_l &:= (n_1,n_2, \ldots n_l) \in \mathbf{I}_l \qquad (l \in \mathbb{N}).
	\end{align}

	If $\mathbf{i}, \mathbf{j} \in \mathbf{I}$, then we define $\mathbf{ij} \in \mathbf{I}$ by concatenation
	\begin{align}
	\mathbf{ij} := \begin{cases} \mathbf{i} &\text{ if } \mathbf{j} = \emptyset, \\
	\mathbf{j} &\text{ if } \mathbf{i} = \emptyset, \\ 
	(i_1,i_2, \ldots, i_m, j_1,j_2, \ldots j_n) &\text{ if } \mathbf{i} = (i_1,i_2, \ldots, i_m) \text{ and } \mathbf{j} = (j_1,j_2, \ldots, j_n). \end{cases}
	\end{align}
	For each $\mathbf{i} \in \mathbf{I}$ we define $s_{\mathbf{i}} \in \text{Cu}_2 \setminus \{\lozenge\}$ by
	\begin{align}
	s_{\mathbf{i}}:= \begin{cases}
	e &\text{ if } \mathbf{i}= \emptyset, \\
	s_{i_1} s_{i_2} \cdots s_{i_n} &\text{ if } \mathbf{i}= (i_1,i_2, \ldots i_n) \in \mathbf{I} \setminus \{\emptyset\}.
	\end{cases}
	\end{align}
\end{definition}

We clearly have $s_{\mathbf{i}} s_{\mathbf{j}} = s_{\mathbf{ij}}$ and $s_{\mathbf{ij}}^*:= (s_{\mathbf{ij}})^* =(s_{\mathbf{i}} s_{\mathbf{j}})^* = s_{\mathbf{j}}^* s_{\mathbf{i}}^*$.

\subsection{Basic combinatorics of $\text{Cu}_2$}

We collect some combinatorial results, which while known, we state for ease of reference.  The following result, as stated below, can be found in \cite[Lemma~3.7]{dh}, where it is attributed to Cuntz (see \cite[Lemmas~1.2~and~1.3]{cuntz}).

\begin{lemma}\label{l: nuts_and_bolts}
	\begin{enumerate}
		\item[(1)] For every $\mathbf{i},\mathbf{j} \in \mathbf{I}$ we have
		\begin{align}
		s_{\mathbf{i}}^* s_{\mathbf{j}}= \begin{cases}
		s_{\mathbf{k}}^* & \text{ if } \mathbf{i}= \mathbf{jk} \text{ for some } \mathbf{k} \in \mathbf{I}, \\
		s_{\mathbf{k}} & \text{ if } \mathbf{j}= \mathbf{ik} \text{ for some } \mathbf{k} \in \mathbf{I}, \\
		\lozenge & \text{ otherwise.}
		\end{cases}
		\end{align}
		\item[(2)] For every $t \in \text{Cu}_2 \setminus \{\lozenge \}$ there exist unique $\mathbf{i}, \mathbf{j} \in \mathbf{I}$ such that $t=s_{\mathbf{i}} s_{\mathbf{j}}^*$.
	\end{enumerate}
\end{lemma}

\begin{remark}
	Let $t \in \text{Cu}_2 \setminus \{\lozenge \}$. By Lemma~\ref{l: nuts_and_bolts} (2) there exist unique $\mathbf{i}, \mathbf{j} \in \mathbf{I}$ such that $t=s_{\mathbf{i}} s_{\mathbf{j}}^*$. Let $\alpha, \beta \in \mathbb{N}_0$ be the unique numbers such that $\mathbf{i} \in \mathbf{I}_{\alpha}$ and $\mathbf{j} \in \mathbf{I}_{\beta}$. 
	
	Thus we may define the \textit{length} of $t$ as
	\begin{align}
	\text{length}(t):= \alpha + \beta.
	\end{align}
\end{remark}

In fact, Lemma~\ref{l: nuts_and_bolts} (2) features so frequently in our arguments that we shall mostly use it implicitly without referring to it. \smallskip

A very important corollary of the above is the lemma below, which we will use numerous times throughout the rest of the paper.

\begin{lemma}\label{l: cu2_combi}
	Let $\mathbf{i}, \mathbf{j}, \mathbf{m}, \mathbf{n} \in \mathbf{I}$. Then
	\begin{align}\label{f: prodscu2}
	s_{\mathbf{i}}^* s_{\mathbf{m}} s_{\mathbf{n}}^* s_{\mathbf{j}} = \begin{cases}
	s_{\mathbf{qp}}^* & \text{ if } \mathbf{i}= \mathbf{mp} \text{ and } \mathbf{n} = \mathbf{jq} \text{ for some } \mathbf{p}, \mathbf{q} \in \mathbf{I}, \\
	s_{\mathbf{r}}^* & \text{ if } \mathbf{i}= \mathbf{mqr} \text{ and } \mathbf{j} = \mathbf{nq} \text{ for some } \mathbf{r}, \mathbf{q} \in \mathbf{I},  \\
	s_{\mathbf{p}} s_{\mathbf{q}}^* & \text{ if } \mathbf{m}= \mathbf{ip} \text{ and } \mathbf{n} = \mathbf{jq} \text{ for some } \mathbf{p}, \mathbf{q} \in \mathbf{I}, \\
	s_{\mathbf{pq}} & \text{ if } \mathbf{m}= \mathbf{ip} \text{ and } \mathbf{j} = \mathbf{nq} \text{ for some } \mathbf{p}, \mathbf{q} \in \mathbf{I}, \\
	s_{\mathbf{r}} & \text{ if } \mathbf{i}= \mathbf{mp} \text{ and } \mathbf{j} = \mathbf{npr} \text{ for some } \mathbf{p}, \mathbf{r} \in \mathbf{I}, \\
	\lozenge &\text{ otherwise.}
	\end{cases}
	\end{align}
	Consequently, $s_{\mathbf{i}}^* s_{\mathbf{m}} s_{\mathbf{n}}^* s_{\mathbf{j}} = e$ if and only if $\mathbf{i}= \mathbf{mk}$ and $\mathbf{j}= \mathbf{nk}$ for some $\mathbf{k} \in \mathbf{I}$.
\end{lemma}

\begin{proof}
	Applying Lemma~\ref{l: nuts_and_bolts} to $s_{\mathbf{i}}^* s_{\mathbf{m}}$ and $s_{\mathbf{n}}^* s_{\mathbf{j}}$, we immediately obtain that
	\begin{align}
	s_{\mathbf{i}}^* s_{\mathbf{m}} s_{\mathbf{n}}^* s_{\mathbf{j}} = \begin{cases}
	s_{\mathbf{p}}^* s_{\mathbf{q}}^* = s_{\mathbf{qp}}^* & \text{ if } \mathbf{i}= \mathbf{mp} \text{ and } \mathbf{n} = \mathbf{jq} \text{ for some } \mathbf{p}, \mathbf{q} \in \mathbf{I}, \\
	s_{\mathbf{p}}^* s_{\mathbf{q}} & \text{ if } \mathbf{i}= \mathbf{mp} \text{ and } \mathbf{j} = \mathbf{nq} \text{ for some } \mathbf{p}, \mathbf{q} \in \mathbf{I},  \\
	s_{\mathbf{p}} s_{\mathbf{q}}^* & \text{ if } \mathbf{m}= \mathbf{ip} \text{ and } \mathbf{n} = \mathbf{jq} \text{ for some } \mathbf{p}, \mathbf{q} \in \mathbf{I}, \\
	s_{\mathbf{p}} s_{\mathbf{q}} = s_{\mathbf{pq}} & \text{ if } \mathbf{m}= \mathbf{ip} \text{ and } \mathbf{j} = \mathbf{nq} \text{ for some } \mathbf{p}, \mathbf{q} \in \mathbf{I}, \\
	\lozenge &\text{ otherwise.}
	\end{cases}
	\end{align}
	Once more we apply Lemma~\ref{l: nuts_and_bolts} to $s_{\mathbf{p}}^* s_{\mathbf{q}}$, which yields the desired formula \eqref{f: prodscu2}.
	
	The ``consequently'' part follows from inspecting the cases in the above formula and from observing that
	\begin{itemize}
		\item $s_{\mathbf{qp}}^* = e$ or $s_{\mathbf{pq}} = e$ or $s_{\mathbf{p}} s_{\mathbf{q}}^* = e$ if and only if $\mathbf{p}= \emptyset = \mathbf{q}$ if and only if $\mathbf{i}= \mathbf{m}$ and $\mathbf{j}= \mathbf{n}$,
		\item $s_{\mathbf{r}}^* = e$ if and only if $\mathbf{r}= \emptyset$ if and only if $\mathbf{i}=\mathbf{mq}$ and $\mathbf{j}= \mathbf{nq}$,
		\item $s_{\mathbf{r}}=e$ if and only if $\mathbf{r}= \emptyset$ if and only if $\mathbf{i}= \mathbf{mp}$ and $\mathbf{j}= \mathbf{np}$.
	\end{itemize}
\end{proof}

\subsection{A purely infinite quotient of $(\ell^1(\text{Cu}_2 \setminus \{ \lozenge \}), \#)$}

From now on we let $ \mc A := (\ell^1(\text{Cu}_2 \setminus \{ \lozenge \}), \#)$. In this section we study a natural quotient of $\mc A$, which is related to the Leavitt algebra $L_2$ (see Remark~\ref{rem:leavitt_algs}), and show that the identity of this quotient factors through every non-zero element of the quotient.

\begin{remark}\label{rem:cohn_algs}
Suppose we start instead with the group ring $\mathbb C[\text{Cu}_2]$, which is just the algebra of finitely supported
elements of $\ell^1(\text{Cu}_2)$, and similarly quotient by the span of $\delta_{\lozenge}$.  As observed in \cite[Section~1]{p1}, the algebra $\mathbb C[\text{Cu}_2] / \mathbb C\delta_\lozenge$ was studied, with a different
presentation, by Cohn in \cite[Section~5]{cohn}, and is sometimes called the \emph{Cohn algebra $C_2$}.

We could hence view $\mc A$ as being a Banach algebra completion of $C_2$.  To our knowledge, this algebra has not been studied from this perspective; for example, it is not mentioned in \cite{dlr}.  We make remarks about links, or lack thereof, with Phillips's work in \cite{p1} below, Remark~\ref{rem:phillips}.
\end{remark}

Let us observe first that in view of Lemma~\ref{l: nuts_and_bolts}, we may write
\begin{align}\label{eq: elements_of_a}
f= \sum\limits_{t \in \text{Cu}_2 \setminus \{\lozenge \}} f(t) \delta_t = \sum\limits_{\mathbf{i}, \mathbf{j} \in \mathbf{I}} f(s_{\mathbf{i}} s_{\mathbf{j}}^*) \delta_{s_{\mathbf{i}} s_{\mathbf{j}}^*} \qquad (f \in \mc A).
\end{align}

Our first goal is to find a useful sufficient condition which guarantees that an element $f \in \mc A$ is purely infinite, in other words, that there exist $g,h \in \mc A$ with $g \# f \# h= \delta_e$.

\begin{definition} \
	\begin{itemize}
		\item Let $v \in \text{Cu}_2 \setminus \{\lozenge \}$, and let $\mathbf{i}, \mathbf{j}$ be the unique elements in $\mathbf{I}$ with $v= s_{\mathbf{i}} s_{\mathbf{j}}^*$. 
		\begin{itemize}
			\item Suppose $\mathbf{n} \in \mathbf{L}$. We define
			\begin{align}\label{eq: pre_zero_sum}
			v_l^{\mathbf{n}}:= s_{\mathbf{in}_l} s_{\mathbf{jn}_l}^* = s_{\mathbf{i}}s_{(n_1, \ldots, n_l)} s_{(n_1, \ldots, n_l)}^* s_{\mathbf{j}}^* \qquad (l \in \mathbb{N}_0).
			\end{align}
			\item Suppose $\mathbf{n} \in \mathbf{I}$. There is a unique $\alpha \in \mathbb{N}_0$ satisfying $\mathbf{n} \in \mathbf{I}_{\alpha}$; hence $\mathbf{n} = (n_1, n_2, \ldots , n_{\alpha})$.
			We define $v_l^{\mathbf{n}}$ as in \eqref{eq: pre_zero_sum} provided $l \in \mathbb{N}_0$ is such that $l \leqslant \alpha$. Otherwise $v_l^{\mathbf{n}}$ is undefined.
		\end{itemize}
		We have in particular $v_0^{\mathbf{n}}= s_{\mathbf{in}_0} s_{\mathbf{jn}_0}^* = s_{\mathbf{i}} s_{\mathbf{j}}^*=v$, and that $e^{\mathbf{n}}_l = s_{(n_1, \ldots, n_l)} s_{(n_1, \ldots, n_l)}^*$.
		\item We say that $f \in \mc A$ \emph{has zero sums at $v =s_{\mathbf{i}} s_{\mathbf{j}}^* \in \text{Cu}_2 \setminus \{\lozenge\}$} if
		\begin{align}
		\sum\limits_{l \in \mathbb{N}_0} f(v_l^{\mathbf{n}})= \sum\limits_{l \in \mathbb{N}_0} f(s_{\mathbf{in}_l} s_{\mathbf{jn}_l}^*)= 0 \qquad (\mathbf{n} \in \mathbf{L}).\label{eq:one}
		\end{align}
	\end{itemize}
\end{definition}

Notice that as $f$ is an $\ell^1$ element, the sum in (\ref{eq:one}) is absolutely convergent.

\begin{lemma}\label{l: premain1}
	Let $\mathbf{n} \in \mathbf{I} \cup \mathbf{L}$ and $f \in \mc A$. Then
	\begin{align}
	\delta_{s_{\mathbf{n}_l}^*} \# f \# \delta_{s_{\mathbf{n}_l}} = \left( \sum\limits_{k=0}^l f(e_k^{\mathbf{n}}) \right) \delta_e + \sum\limits_{\substack{ \mathbf{i}, \mathbf{j} \in \mathbf{I} \\ s_{\mathbf{n}_l}^* s_{\mathbf{i}} s_{\mathbf{j}}^* s_{\mathbf{n}_l} \notin \{e, \lozenge \}}} f(s_{\mathbf{i}} s_{\mathbf{j}}^* ) \delta_{s_{\mathbf{n}_l}^* s_{\mathbf{i}} s_{\mathbf{j}}^* s_{\mathbf{n}_l}} \qquad (l \in \mathbb{N}_0).
	\end{align}
\end{lemma}

\begin{proof}
	Let us fix an $l \in \mathbb{N}_0$. We first note that by Lemma~\ref{l: cu2_combi}
	\begin{align*}
	\{s_{\mathbf{i}} s_{\mathbf{j}}^* \colon \mathbf{i}, \mathbf{j} \in \mathbf{I}, \; s_{\mathbf{n}_l}^* s_{\mathbf{i}} s_{\mathbf{j}}^* s_{\mathbf{n}_l} = e \} &= \{s_{\mathbf{i}} s_{\mathbf{j}}^* \colon \mathbf{i}, \mathbf{j} \in \mathbf{I}, \; \mathbf{n}_l = \mathbf{ip}, \; \mathbf{n}_l = \mathbf{jp} \text{ for some } \mathbf{p} \in \mathbf{I} \} \\
	&= \{s_{\mathbf{i}} s_{\mathbf{i}}^* \colon \mathbf{i}= \mathbf{n}_k \text{ for some } 0 \leqslant k \leqslant l \} \\
	&= \{e_{k}^{\mathbf{n}} \colon 0 \leqslant k \leqslant l \}.
	\end{align*}
	Since
	\begin{align*}
	\delta_{s_{\mathbf{n}_l}^*} \# f \# \delta_{s_{\mathbf{n}_l}} &= \sum\limits_{\substack{ \mathbf{i}, \mathbf{j} \in \mathbf{I} \\ s_{\mathbf{n}_l}^* s_{\mathbf{i}} s_{\mathbf{j}}^* s_{\mathbf{n}_l} \neq \lozenge}} f(s_{\mathbf{i}} s_{\mathbf{j}}^* ) \delta_{s_{\mathbf{n}_l}^* s_{\mathbf{i}} s_{\mathbf{j}}^* s_{\mathbf{n}_l}},
	\end{align*}
	the result follows.
\end{proof}

\begin{proposition}\label{p: main1}
	Let $f \in \mathcal{A}$ be such that it does not have zero sums at the multiplicative unit $e \in \text{Cu}_2$. Then there exist $g,h \in \mc A$ with $g \# f \# h = \delta_e$.
\end{proposition}

\begin{proof}
	By the assumption there is an $\mathbf{n} = (n_1,n_2, \ldots, n_k, \ldots ) \in \mathbf{L}$ such that $\sum_{k \in \mathbb{N}_0} f(e_k^{\mathbf{n}}) \neq 0$. Let us set $z_N := \sum_{k=0}^N f(e_k^{\mathbf{n}})$ for each $N \in \mathbb{N}_0$. As $f \in \mc A$, the sequence $(z_N)$ converges to some non-zero element in $\mathbb{C}$, therefore there is an $\varepsilon >0$ and $N' \in \mathbb{N}_0$ such that $|z_n| \geqslant 2 \varepsilon$ for each $n \geqslant N'$. From Lemma~\ref{l: premain1} we see that
	\begin{align}
	\delta_{s_{\mathbf{n}_l}^*} \# f \# \delta_{s_{\mathbf{n}_l}} &= z_l \delta_e + \sum\limits_{\substack{ \mathbf{i}, \mathbf{j} \in \mathbf{I} \\ s_{\mathbf{n}_l}^* s_{\mathbf{i}} s_{\mathbf{j}}^* s_{\mathbf{n}_l} \notin \{e, \lozenge \}}} f(s_{\mathbf{i}} s_{\mathbf{j}}^*) \delta_{s_{\mathbf{n}_l}^* s_{\mathbf{i}} s_{\mathbf{j}}^* s_{\mathbf{n}_l}} \qquad (l \in \mathbb{N}_0).
	\end{align}
	There is a $K \in \mathbb{N}$ such that $\sum_{\mathbf{i}, \mathbf{j} \in \mathbf{I} \setminus \cup_{k=0}^M \mathbf{I}_k} |f(s_{\mathbf{i}} s_{\mathbf{j}}^*) | < \varepsilon$ for all $M \geqslant K$. Let us fix an $M \geqslant \max \{N', K\}$ and define
	\begin{align}
	f':= \sum\limits_{\mathbf{i}, \mathbf{j} \in \cup_{k=0}^M \mathbf{I}_k} f(s_{\mathbf{i}} s_{\mathbf{j}}^*) \delta_{s_{\mathbf{i}} s_{\mathbf{j}}^*}.
	\end{align}
	Clearly, $f' \in \mc A$ is finitely supported such that $\|f-f'\| < \varepsilon$. Also, $f'(t)=0$ whenever $t \in \text{Cu}_2$ is such that $\text{length}(t) >2M$. We also note that $f'(e_k^{\mathbf{n}}) = f'(s_{\mathbf{n}_k} s_{\mathbf{n}_k}^*) =f(s_{\mathbf{n}_k} s_{\mathbf{n}_k}^*) = f(e_k^{\mathbf{n}})$ for each $k \in \{1, \ldots, M\}$ by the definition of $f'$. Consequently,
	\begin{align}\label{e: z_is_what_we_want}
	z_M = \sum\limits_{k=0}^M f(e_k^{\mathbf{n}}) = \sum\limits_{k=0}^M f'(e_k^{\mathbf{n}}).
	\end{align}
	To ease notation, we put $z:=z_M$.
	
	\begin{claim}\label{cl: iterative}
		There is a $\mathbf{p} \in \mathbf{I}$ such that
		\begin{align}
		\delta_{s_{\mathbf{p}}^*} \# \delta_{s_{\mathbf{n}_M}^*} \# f' \# \delta_{s_{\mathbf{n}_M}} \# \delta_{s_{\mathbf{p}}} = \delta_{s_{\mathbf{n}_M \mathbf{p}}^*} \# f' \# \delta_{s_{\mathbf{n}_M \mathbf{p}}} &= z \delta_e.
		\end{align}
	\end{claim}	
	\begin{proof}[Proof of Claim~\ref{cl: iterative}]
		By Lemma~\ref{l: premain1} and \eqref{e: z_is_what_we_want} we have $\delta_{s_{\mathbf{n}_M}^*} \# f' \# \delta_{s_{\mathbf{n}_M}} = z \delta_e + h_0$, where
		\begin{align}
		h_0 := \sum\limits_{\substack{ \mathbf{i}, \mathbf{j} \in \mathbf{I} \\ s_{\mathbf{n}_M}^* s_{\mathbf{i}} s_{\mathbf{j}}^* s_{\mathbf{n}_M} \notin \{e, \lozenge \}}} f'(s_{\mathbf{i}} s_{\mathbf{j}}^*) \delta_{s_{\mathbf{n}_M}^* s_{\mathbf{i}} s_{\mathbf{j}}^* s_{\mathbf{n}_M}}.
		\end{align}
		If $h_0=0$ then we are done. Otherwise, $H_0 := \support(h_0) \neq \emptyset$. We claim that there is an $i \in \{1,2\}$ such that
		\begin{align}
		\left| \{ s_i^* t s_i \colon t \in H_0, \; s_i^* t s_i \neq \lozenge \} \right| < |H_0|.
		\end{align}
		To show this, observe that as $\support(f')$, and hence also $H_0$, is finite, it is enough to see that $s_1^*ts_1 = \lozenge$ or $s_2^*ts_2 = \lozenge$ for some $t \in H_0$. This readily follows however from $\{e\} \neq H_0$ (which clearly holds as $e \notin H_0$).
		
		For this choice of $i$, applying Lemma~\ref{l: premain1} again we see that
		\begin{align}
		\delta_{s_{\mathbf{n}_M i}^*} \# f' \# \delta_{s_{\mathbf{n}_M} i} &= \delta_{s_i^*} \# (z \delta_e +h_0) \# \delta_{s_i} = z \delta_e + \delta_{s_i^*} \# h_0 \# \delta_{s_i} \notag \\
		&= (z + h_0(e) + h_0(s_i s_i^*)) \delta_e  + h_1,
		\end{align}
		where
		\begin{align}
		h_1 := \sum\limits_{\substack{ \mathbf{i}, \mathbf{j} \in \mathbf{I} \\ s_i^* s_{\mathbf{i}} s_{\mathbf{j}}^* s_i \notin \{e, \lozenge \}}} h_0(s_{\mathbf{i}} s_{\mathbf{j}}^*) \delta_{s_i^* s_{\mathbf{i}} s_{\mathbf{j}}^* s_i}.
		\end{align}
		Note that $\support(h_1) \subseteq \{ s_i^* t s_i \colon t \in H_0, \; s_i^* t s_i \neq \lozenge \}$.
		
		On the one hand $h_0(e)=0$. On the other hand $s_{\mathbf{n}_M}^* s_{\mathbf{i}} s_{\mathbf{j}}^* s_{\mathbf{n}_M} = s_i s_i^*$ if and only if $\mathbf{i}= \mathbf{n}_M i$ and $\mathbf{j}= \mathbf{n}_M i$ by Lemma~\ref{l: cu2_combi}, hence
		\begin{align}
		h_0(s_i s_i^*) =  f'(s_{\mathbf{n}_M i} s_{\mathbf{n}_M i}^*) =0.
		\end{align}
		The last equality follows because $\text{length}(s_{\mathbf{n}_M i} s_{\mathbf{n}_M i}^*) = 2(M+1)$, and $f'$ vanishes on elements of $\text{Cu}_2$ of length at least $2M+1$.  Consequently,
		\begin{align}
		\delta_{s_{\mathbf{n}_M i}^*} \# f' \# \delta_{s_{\mathbf{n}_M} i} = z \delta_e  + h_1,
		\end{align}
		where $H_1:= \support(h_1)$ is such that $|H_1| < |H_0|$.
		
		Let us fix some $k_0 > | \support(f')|$. Continuing recursively, we obtain $i_1,i_2, \ldots, i_{k_0} \in \{ 1,2 \}$ and finitely supported functions $(h_k)_{k=1}^{k_0}$ in $\mc A$ with $H_k:= \support(h_k)$ such that
		\begin{align}
		\delta_{s_{\mathbf{n}_M (i_1, \ldots, i_{k})}^*} &\# f' \# \delta_{s_{\mathbf{n}_M} (i_1, \ldots, i_{k})} = z \delta_e  + h_k \qquad (1 \leqslant k \leqslant k_0), \\
		|H_0| &> |H_1| > \ldots > |H_{k_0}|.
		\end{align}
		As $\support(f')$ is finite, we must have that $H_{k_0}= \emptyset$ or equivalently $h_{k_0}=0$. Thus setting $\mathbf{p}:= (i_1, \ldots, i_{k_0}) \in \mathbf{I}$ yields the claim.
	\end{proof}
	
	We now finish the main proof. From the claim  we obtain
	\begin{align}
	\| \delta_e - z^{-1} \delta_{s_{\mathbf{n}_M \mathbf{p}}^*} \# f \# \delta_{s_{\mathbf{n}_M \mathbf{p}}} \| &= |z|^{-1} \| \delta_{s_{\mathbf{n}_M \mathbf{p}}^*} \# (f'- f) \# \delta_{s_{\mathbf{n}_M \mathbf{p}}} \| \notag \\
	&\leqslant |z|^{-1} \|f-f' \| < 1/2,
	\end{align}
	thus the Carl Neumann series implies $u:=z^{-1} \delta_{s_{\mathbf{n}_M \mathbf{p}}^*} \# f \# \delta_{s_{\mathbf{n}_M \mathbf{p}}} \in \inv(\mc A)$. Hence setting $g:= u^{-1} \# z^{-1} \delta_{s_{\mathbf{n}_M \mathbf{p}}^*}$ and $h:= \delta_{s_{\mathbf{n}_M \mathbf{p}}}$ concludes the proof.
\end{proof}

In the following, let $\mc J$ denote the closed, two-sided ideal in $\mc A$ generated by the element 
\begin{align}
f_0 :=\delta_e - \delta_{s_1 s_1^*} - \delta_{s_2 s_2^*}.
\end{align}
Clearly $f_0$ is an projection in $\mc A$, in other words, $f_0^2 = f_0$ and $f_0^* = f_0$. We immediately see from the formula~\eqref{eq: elements_of_a} and Lemma~\ref{l: nuts_and_bolts} (2) that
\begin{align}\label{eq: elements_of_j_1}
\mc J &= \overline{\spanning} \{g \# f_0 \# h \colon g,h \in \mc A \} = \overline{\spanning} \{\delta_{s_{\mathbf{i}} s_{\mathbf{k}}^*} \# f_0 \# \delta_{s_{\mathbf{l}} s_{\mathbf{j}}^*} \colon \mathbf{i}, \mathbf{j}, \mathbf{k}, \mathbf{l} \in \mathbf{I} \}.
\end{align}

Since $f_0^* = f_0$, it follows from \eqref{eq: elements_of_j_1} that  $\mc J$ is a $*$-ideal in $\mc A$, and therefore $\mc A / \mc J$ is a Banach $*$-algebra.

\begin{remark}\label{rem:leavitt_algs}
Continuing Remark~\ref{rem:cohn_algs}, in the Cohn Algebra $C_2 \cong \mathbb C[\text{Cu}_2]/\mathbb C\delta_\lozenge$ we could also consider the ideal, say $J_2$, generated by $f_0$.  Then $C_2 / J_2$ is seen to be isomorphic to the \emph{Leavitt algebra $L_2$}, see \cite[Section~1]{p1}, which was first considered (over the field with $2$ elements) in \cite{leavitt}.

$\mc A/\mc J$ is a Banach algebraic completion of $L_2$, which again seems not to have been considered in the literature before. Compare with Remark~\ref{rem:phillips} below.
\end{remark}

Let us introduce some new terminology which will render the technical proofs in this section significantly more transparent.

\begin{definition} \
	\begin{enumerate}
		\item[(i)] An element $t=s_{\mathbf{i}} s_{\mathbf{j}}^* \in \text{Cu}_2 \setminus \{ \lozenge \}$ is \emph{symmetric} if $\mathbf{i} = \mathbf{j}$.
		\item[(ii)] We say that $t \in \text{Cu}_2 \setminus \{ \lozenge \}$ is a \emph{symmetric expansion of} $r = s_{\mathbf{m}} s_{\mathbf{n}}^* \in \text{Cu}_2 \setminus \{ \lozenge \}$ if there exists a symmetric $u \in \text{Cu}_2 \setminus \{ \lozenge \}$ with $t=s_{\mathbf{m}}u s_{\mathbf{n}}^*$. If in addition $u \neq e$ then we say that $t$ is a \emph{proper symmetric expansion of} $r$.
		\item[(iii)] For some $t \in \text{Cu}_2 \setminus \{\lozenge\}$ the set of symmetric expansions of $t$ is denoted by $S_t$.
		\item[(iv)] An element of $\text{Cu}_2 \setminus \{ \lozenge \}$ \emph{is without symmetric core} if it is not the proper symmetric expansion of any element in $\text{Cu}_2 \setminus \{ \lozenge \}$.
	\end{enumerate}
\end{definition}

The following are immediate from the definition.

\begin{remark} \
	\begin{itemize}
		\item An element $t \in \text{Cu}_2 \setminus \{ \lozenge \}$ is a symmetric expansion of $r = s_{\mathbf{m}} s_{\mathbf{n}}^* \in \text{Cu}_2 \setminus \{ \lozenge \}$ if and only if there exists $\mathbf{i} \in \mathbf{I}$ with $t= s_{\mathbf{mi}} s_{\mathbf{ni}}^*$. Also, $t$ is a proper expansion of $r$ if and only if $\mathbf{i} \neq \emptyset$.
		\item An element $t \in \text{Cu}_2 \setminus \{ \lozenge \}$ is without symmetric core if and only if whenever $\mathbf{m}, \mathbf{n}, \mathbf{i} \in \mathbf{I}$ are such that $t=s_{\mathbf{mi}} s_{\mathbf{ni}}^*$ then $\mathbf{i} = \emptyset$.
	\end{itemize}
\end{remark}

\begin{lemma}\label{l: partition}
	The set
	\begin{align}
	\{ S_v \colon v \in \text{Cu}_2 \setminus \{\lozenge\} \text{ is without symmetric core}\} 
	\end{align}
	forms a partition of $\text{Cu}_2 \setminus \{\lozenge\}$.
\end{lemma}

\begin{proof}
	Let $t \in \text{Cu}_2 \setminus \{\lozenge\}$ be arbitrary. There exist unique $\mathbf{p}, \mathbf{q} \in \mathbf{I}$ such that $t=s_{\mathbf{p}} s_{\mathbf{q}}^*$. Let $\alpha \in \mathbb{N}_0$ be maximal with respect to the property that there is an $\mathbf{i} \in \mathbf{I}_{\alpha}$ with $\mathbf{p}= \mathbf{mi}$ and $\mathbf{q}= \mathbf{ni}$ for some $\mathbf{m}, \mathbf{n} \in \mathbf{I}$. Then $t=s_{\mathbf{p}} s_{\mathbf{q}}^* = s_{\mathbf{mi}} s_{\mathbf{ni}}^* = s_{\mathbf{m}} (s_{\mathbf{i}} s_{\mathbf{i}}^*) s_{\mathbf{n}}^*$ shows that $t$ is the symmetric expansion of $v:=s_{\mathbf{m}} s_{\mathbf{n}}^*$. Observe that $v$ is without symmetric core. For assume towards a contradiction that there exists $\mathbf{k} \in \mathbf{I} \setminus \mathbf{I}_0$ such that $v= s_{\mathbf{ak}}s_{\mathbf{bk}}^*$ for some $\mathbf{a}, \mathbf{b} \in \mathbf{I}$. Therefore $\mathbf{m}=\mathbf{ak}$ and $\mathbf{n}= \mathbf{bk}$ must hold, consequently $t=s_{\mathbf{mi}} s_{\mathbf{ni}}^* = s_{\mathbf{aki}} s_{\mathbf{bki}}^*$. This contradicts the maximality of $\alpha$.
	
	Let $v,w \in \text{Cu}_2 \setminus \{\lozenge\}$ be without symmetric core. Assume there is some $t \in S_v \cap S_w$. Let $\mathbf{i}, \mathbf{j} \in \mathbf{I}$ be unique with $v=s_{\mathbf{i}} s_{\mathbf{j}}^*$, then $t=s_{\mathbf{ik}} s_{\mathbf{jk}}^*$ for some $\mathbf{k} \in \mathbf{I}$. Similarly, let $\mathbf{p}, \mathbf{q} \in \mathbf{I}$ be unique with $w=s_{\mathbf{p}} s_{\mathbf{q}}^*$, then $t=s_{\mathbf{pl}} s_{\mathbf{ql}}^*$ for some $\mathbf{l} \in \mathbf{I}$. As $s_{\mathbf{ik}} s_{\mathbf{jk}}^* =t=s_{\mathbf{pl}} s_{\mathbf{ql}}^*$, it follows that $\mathbf{ik}=\mathbf{pl}$ and $\mathbf{jk}=\mathbf{ql}$. We want to show that $v=w$, equivalently $\mathbf{i}=\mathbf{p}$ and $\mathbf{j}= \mathbf{q}$. Let $\alpha, \beta \in \mathbb{N}_0$ be the unique numbers such that $\mathbf{k} \in \mathbf{I}_{\alpha}$ and $\mathbf{l} \in \mathbf{I}_{\beta}$. Note that it is enough to show that $\alpha=\beta$. Assume towards a contradiction that, say, $\alpha < \beta$. Then there are $\mathbf{m} \in \mathbf{I}_{\beta - \alpha}$ and $\mathbf{n} \in \mathbf{I}_{\alpha}$ with $\mathbf{l} = \mathbf{mn}$. Thus $\mathbf{ik}= \mathbf{pmn}$, hence from $\mathbf{k}, \mathbf{n} \in \mathbf{I}_{\alpha}$ we obtain $\mathbf{i}= \mathbf{pm}$. Similarly, we get $\mathbf{j}= \mathbf{qm}$. But then $v = s_{\mathbf{i}} s_{\mathbf{j}}^* = s_{\mathbf{pm}} s_{\mathbf{qm}}^*$, which by $\beta - \alpha >0$ contradicts that $v$ is without symmetric core.
\end{proof}

\begin{proposition}\label{p: main2}
	Suppose $f \in \mc A$ is such that it has zero sums at some $v \in \text{Cu}_2 \setminus \{\lozenge \}$. Define $h := \sum_{t \in S_v} f(t) \delta_t \in \mc A$. Then $h \in \mathcal{J}$.
\end{proposition}

\begin{proof}
	Let $\mathbf{i}, \mathbf{j} \in \mathbf{I}$ be such that $v=s_{\mathbf{i}} s_{\mathbf{j}}^*$. Then $S_v = \{ s_{\mathbf{ik}} s_{\mathbf{jk}}^* \colon \mathbf{k} \in \mathbf{I} \}$ and hence $h= \sum_{\mathbf{k} \in \mathbf{I}} f(s_{\mathbf{ik}} s_{\mathbf{jk}}^*) \delta_{s_{\mathbf{ik}} s_{\mathbf{jk}}^*}$. From $f_0 \in \mc J$ we immediately get
	\begin{align}\label{eq: inideal}
	\delta_{s_{\mathbf{im}} s_{\mathbf{jm}}^*} - \delta_{s_{\mathbf{im}1} s_{\mathbf{jm}1}^*} - \delta_{s_{\mathbf{im}2} s_{\mathbf{jm}2}^*} = \delta_{s_{\mathbf{im}}} \# f_0 \#  \delta_{s_{\mathbf{jm}}^*} \in \mc J \qquad (\mathbf{m} \in \mathbf{I}).
	\end{align}
	In particular, setting $\mathbf{m}:= \emptyset$ in \eqref{eq: inideal} yields
	\begin{align}\label{eq: instep1}
	\delta_{v} - \delta_{s_{\mathbf{i}1} s_{\mathbf{j}1}^*} - \delta_{s_{\mathbf{i}2} s_{\mathbf{j}2}^*} \in \mc J.
	\end{align}
	Hence from 
	\begin{align}
	h= f(v) \delta_v + f(s_{\mathbf{i}1} s_{\mathbf{j}1}^*) \delta_{s_{\mathbf{i}1} s_{\mathbf{j}1}^*} + f(s_{\mathbf{i}2} s_{\mathbf{j}2}^*) \delta_{s_{\mathbf{i}2} s_{\mathbf{j}2}^*} + \sum_{\mathbf{k} \in \mathbf{I} \setminus (\mathbf{I}_0 \cup \mathbf{I}_1)} f(s_{\mathbf{ik}} s_{\mathbf{jk}}^*) \delta_{s_{\mathbf{ik}} s_{\mathbf{jk}}^*}
	\end{align}
	and \eqref{eq: instep1} we see that
	\begin{align}\label{eq: improvedeq}
	\big(f(v) &+ f(s_{\mathbf{i}1} s_{\mathbf{j}1}^*) \big) \delta_{s_{\mathbf{i}1} s_{\mathbf{j}1}^*} + \big(f(v) + f(s_{\mathbf{i}2} s_{\mathbf{j}2}^*) \big) \delta_{s_{\mathbf{i}2} s_{\mathbf{j}2}^*} \notag \\
	&+ \sum_{\mathbf{k} \in \mathbf{I} \setminus (\mathbf{I}_0 \cup \mathbf{I}_1)} f(s_{\mathbf{ik}} s_{\mathbf{jk}}^*) \delta_{s_{\mathbf{ik}} s_{\mathbf{jk}}^*} -h \in \mc J \notag \\
	\notag \\
	\Longleftrightarrow \sum_{\mathbf{k} \in \mathbf{I}_1} \big(f(v) &+ f(s_{\mathbf{i}} e_1^{\mathbf{k}} s_{\mathbf{j}}^*) \big) \delta_{s_{\mathbf{i}} e_1^{\mathbf{k}} s_{\mathbf{j}}^*}  \notag \\
	&+ \sum_{\mathbf{k} \in \mathbf{I} \setminus (\mathbf{I}_0 \cup \mathbf{I}_1)} f(s_{\mathbf{ik}} s_{\mathbf{jk}}^*) \delta_{s_{\mathbf{ik}} s_{\mathbf{jk}}^*} -h \in \mc J.
	\end{align}
	Continuing inductively, we obtain
	\begin{align}\label{eq: almosthere}
	\sum_{\mathbf{k} \in \mathbf{I}_n} \left( f(v) + \sum_{l=1}^n f(s_{\mathbf{i}} e_l^{\mathbf{k}} s_{\mathbf{j}}^*) \right) \delta_{s_{\mathbf{i}} e_n^{\mathbf{k}} s_{\mathbf{j}}^*} + \sum_{\mathbf{k} \in \mathbf{I} \setminus \bigcup_{r=0}^n \mathbf{I}_r} f(s_{\mathbf{ik}} s_{\mathbf{jk}}^*) \delta_{s_{\mathbf{ik}} s_{\mathbf{jk}}^*} -h \in \mc J \qquad (n \in \mathbb{N}).
	\end{align}
	That $f$ has zero sums at $v$ is to say $f(v) = - \sum_{l \in \mathbb{N}} f(s_{\mathbf{i}} e_l^{\mathbf{k}} s_{\mathbf{j}}^*)$ for each $\mathbf{k} \in \mathbf{L}$. Therefore \eqref{eq: almosthere} is equivalent to
	\begin{align}\label{eq: final}
	- \sum_{\mathbf{k} \in \mathbf{I}_n} \left(\sum_{l>n} f(s_{\mathbf{i}} e_l^{\mathbf{k}} s_{\mathbf{j}}^*) \right) \delta_{s_{\mathbf{i}} e_n^{\mathbf{k}} s_{\mathbf{j}}^*} + \sum_{\mathbf{k} \in \mathbf{I} \setminus \bigcup_{r=0}^n \mathbf{I}_r} f(s_{\mathbf{ik}} s_{\mathbf{jk}}^*) \delta_{s_{\mathbf{ik}} s_{\mathbf{jk}}^*} -h \in \mc J \qquad (n \in \mathbb{N}).
	\end{align}
	Now $\sum_{t \in \text{Cu}_2 \setminus \{\lozenge \}} |f(t)| < \infty$ implies
	\begin{align}\label{eq: postfinal}
	\sum_{\mathbf{k} \in \mathbf{I}_n} \left| \sum_{l>n} f(s_{\mathbf{i}} e_l^{\mathbf{k}} s_{\mathbf{j}}^*) \right| \rightarrow 0 \quad \text{and} \quad \sum_{\mathbf{k} \in \mathbf{I} \setminus \bigcup_{r=0}^n \mathbf{I}_r} |f(s_{\mathbf{ik}} s_{\mathbf{jk}}^*)| \rightarrow 0 \qquad (n \rightarrow \infty).
	\end{align}
	As $\mc J$ is closed, we conclude from \eqref{eq: postfinal} and \eqref{eq: final} that $h \in \mc J$.
\end{proof}

\subsection{Representing $\mc A / \mc J$ in $\mc B (\ell^p)$}

To make further progress on understanding $\mc A/\mc J$, we will consider certain representations of this algebra, thus proving in particular that it is non-trivial.  From now on we let $\pi_{\mc J} \colon \mc A \rightarrow \mc A / \mc J$ denote the quotient map.

We will represent $\mc A/\mc J$ inside $\mc B (\ell^p(\mathbb{N}))$, the unital Banach algebra of bounded linear operators on $\ell^p(\mathbb{N})$, for any $p \in [1, \infty)$.
Let $p \in [1, \infty)$ be fixed. We first define operators $A_1, A_2, B_1, B_2$ on $\ell^p:= \ell^p(\mathbb{N})$ by
\begin{align}
(A_1 x)(n) = x_{2n}, \qquad (A_2 x)(n) = x_{2n-1} \qquad (x \in \ell^p),
\end{align}
and
\begin{align}
(B_1 x)(n) = \begin{cases} x_{n/2} &\text{if } n \in 2 \mathbb{N}, \\ 0 & \text{otherwise,} \end{cases}
\quad
(B_2 x)(n) = \begin{cases} x_{(n+1)/2} &\text{if } n \in 2 \mathbb{N}-1, \\ 0 & \text{otherwise.} \end{cases}
\qquad (x \in \ell^p).
\end{align}

It is immediate that $A_i, B_i, \in \mc B (\ell^p)$ with $\|A_i \|=1= \|B_i \|$ for $i \in \{1,2\}$. Moreover, the following relations hold:
\begin{align}
A_1 B_1 = I_{\ell^p} = A_2 B_2, \quad A_1 B_2 = 0 = A_2 B_1, \quad B_1 A_1 + B_2 A_2 = I_{\ell^p},
\end{align}
where $I_{\ell^p}$ denotes the identity operator on $\ell^p$.

\begin{remark}\label{linindep}
	Let us note that the set $\{B_1^n \colon n \in \mathbb{N}\}$ is linearly independent in $\mc B (\ell^p)$. Indeed, suppose $(\alpha_n)_{n=1}^N$ is a finite family of scalars such that $\sum_{n=1}^N \alpha_n B_1^n =0$. Let $(e_n)$ be the standard unit vector basis of $\ell^p$. We see that
	\begin{align}
	0= \sum_{n=1}^N \alpha_n B_1^n e_1 = \sum_{n=1}^N \alpha_n e_{2^n},
	\end{align}
	hence $\alpha_n =0$ must hold whenever $1 \leqslant n \leqslant N$.
\end{remark}

\begin{proposition}\label{p: rep_quot}
	For each $p \in [1, \infty)$ there is a continuous, unital algebra homomorphism $\Theta_p \colon \mathcal{A} / \mathcal{J} \rightarrow \mathcal{B}(\ell^p)$ with 
	\begin{align}
	\Theta_p(\pi_{\mc J}(\delta_{s_i^*}))=A_i \quad \text{and} \quad \Theta_p(\pi_{\mc J}(\delta_{s_i}))=B_i \qquad (i \in \{1,2\}).
	\end{align}
	In particular $\mc A / \mc J$ is infinite-dimensional and non-commutative.
\end{proposition}

\begin{proof}
	Fix a $p \in [1, \infty)$. The operators $A_1, A_2, B_1, B_2 \in \mc B (\ell^p)$ are subject to the relations $A_1 B_1= I_{\ell^p} = A_2 B_2$ and $A_2 B_1 = 0 =A_1 B_2$, hence there is a unique semigroup homomorphism 
	\begin{align}
	\phi_p \colon \text{Cu}_2 \rightarrow \mc B (\ell^p)
	\end{align}
	which satisfies $\phi_p(s_1^*)=A_1$, $\phi_p(s_1)=B_1$, $\phi_p(s_2^*)=A_2$ and $\phi_p(s_2)=B_2$. Notice that in particular $\phi_p(e)= \phi_p(s_1^* s_1) = \phi_p(s_1^*) \phi_p(s_1) =A_1 B_1 = I_{\ell^p}$ and $\phi_p(\lozenge) = \phi_p(s_1^* s_2)= \phi_p(s_1^*) \phi_p(s_2)=A_1 B_2=0$. By Lemma~\ref{l: nuts_and_bolts} (2), and because the operators $A_1, A_2, B_1, B_2 \in \mc B (\ell^p)$ have norm one, we see that $\|\phi_p(t)\| \leqslant 1$ for every $t \in \text{Cu}_2$.
	
	It follows that there is a unique continuous algebra homomorphism
	\begin{align}
	\theta_p \colon \mc A = (\ell^1(\text{Cu}_2 \setminus \{ \lozenge \}), \#) \rightarrow \mc B (\ell^p)
	\end{align}
	such that $\| \theta_p \| \leqslant 1$ and $\theta_p(\delta_t) = \phi_p(t)$ for all $t \in \text{Cu}_2 \setminus \{ \lozenge\}$.	
	
	In particular $\theta_p$ is unital as $\theta_p(\delta_e)= \phi_p(e) = I_{\ell^p}$. Moreover, from the relation $B_1 A_1+B_2 A_2 = I_{\ell^p}$ we see
	\begin{align}
	\theta_p(f_0) = \theta_p(\delta_e)- \theta_p(\delta_{s_1}) \theta_p(\delta_{s_1^*}) - \theta_p(\delta_{s_2}) \theta_p(\delta_{s_2^*}) = I_{\ell^p} - B_1 A_1 - B_2 A_2 =0,
	\end{align}
	consequently $\mc J \subseteq \Ker(\theta_p)$. Therefore there is a unique continuous algebra homomorphism
	\begin{align}
	\Theta_p \colon \mc A / \mc J \rightarrow \mc B (\ell^p)
	\end{align}
	with $\| \Theta_p \| \leqslant 1$ such that $\Theta_p \circ \pi_{\mc J} = \theta_p$, where $\pi_{\mc J} \colon \mc A \rightarrow \mc A / \mc J$ is the quotient map. Clearly $\Theta_p(\pi_{\mc J}(\delta_t))= \theta_p(\delta_t)= \phi_p(t)$ for each $t \in \text{Cu}_2 \setminus \{ \lozenge \}$. Consequently	the required relations hold.
	
	Let us show that $\mc A / \mc J$ is infinite-dimensional. We observe that
	\begin{align}
	\{B_1^n \colon n \in \mathbb N\} = \{ \Theta_p (\pi_{\mc J}(\delta_{s_1^n})) \colon n \in \mathbb{N}\} \subseteq \Ran(\Theta_p),
	\end{align}
	and hence $\Ran(\Theta_p)$ is infinite-dimensional by Remark \ref{linindep}. From this it readily follows that $\mc A / \mc J$ is infinite-dimensional too.
	
	Finally, it is clear that $\mc A / \mc J$ is non-commutative.
\end{proof}

\begin{remark}
	It is obvious that the continuous homomorphism $\theta_p \colon \mc A \to \mc B(\ell^p)$ in the proof above is not injective for any $p \in [1, \infty)$. We remark in passing however, that it is possible to find (even explicitly construct) a continuous, unital, faithful $*$-homomorphism $\mc A \to \mc B(\ell^2)$; see \cite[Remark~3.16]{dlr}.
\end{remark}

The following is our main result for showing that $\mc A/\mc J$ is purely infinite.

\begin{theorem}\label{t: main_new}
	Let $f \in \mc A$. Then the following are equivalent:
	\begin{enumerate}
		\item\label{t:main_new:one} $f \in \mc J$;
		\item\label{t:main_new:two} $f$ has zero sums at every $v \in \text{Cu}_2 \setminus \{\lozenge \}$ without symmetric core;
		\item\label{t:main_new:three} There are no $g,h \in \mc A$ with $g \# f \# h = \delta_e$.
	\end{enumerate}
	In particular $\mc J = \mc{M}_{\mc A}$ and hence $\mc J$ is the unique maximal ideal in $\mc A$.
\end{theorem}

\begin{proof}
We first show the contrapositive of ((\ref{t:main_new:three}) $\Rightarrow$ (\ref{t:main_new:two})).  So assume the opposite of (\ref{t:main_new:two}), that is, there exists $v= s_{\mathbf{i}} s_{\mathbf{j}}^* \in \text{Cu}_2 \setminus \{ \lozenge \}$ without symmetric core such that $f$ does not have zero sums at $v$.  We \textit{claim} that $\overline{f}:= \delta_{s_{\mathbf{i}}^*} \# f \# \delta_{s_{\mathbf{j}}} \in \mc A$ does not have zero sums at $e$. To see this, let us fix an $\mathbf{n} \in \mathbf{L}$. Using Lemma~\ref{l: cu2_combi} we see that
	\begin{align}
	\sum_{l \in \mathbb{N}_0} \overline{f}(e_l^{\mathbf{n}}) &= \sum_{l \in \mathbb{N}_0} \overline{f}(s_{\mathbf{n}_l} s_{\mathbf{n}_l}^*) = \sum_{l \in \mathbb{N}_0} (\delta_{s_{\mathbf{i}}^*} \# f \# \delta_{s_{\mathbf{j}}})(s_{\mathbf{n}_l} s_{\mathbf{n}_l}^*) \notag \\
	&= \sum_{l \in \mathbb{N}_0} \left( \sum\limits_{\mathbf{p}, \mathbf{q} \in \mathbf{I}} f(s_{\mathbf{p}} s_{\mathbf{q}}^*) \delta_{s_{\mathbf{i}}^* s_{\mathbf{p}} s_{\mathbf{q}}^* s_{\mathbf{j}}} \right) (s_{\mathbf{n}_l} s_{\mathbf{n}_l}^*) = \sum_{l \in \mathbb{N}_0} f(s_{\mathbf{in}_l} s_{\mathbf{jn}_l}^*) = \sum_{l \in \mathbb{N}_0} f(v_l^{\mathbf{n}}),
	\end{align}
	hence the claim follows because $f$ does not have zero sums at $v$. We can thus apply Proposition~\ref{p: main1}; there exist $\overline{g}, \overline{h} \in \mc A$ with $\delta_e = \overline{g} \# \overline{f} \# \overline{h}$. Consequently $\delta_e = (\overline{g} \# \delta_{s_{\mathbf{i}}^*}) \# f \# (\delta_{s_{\mathbf{j}}} \# \overline{h})$, verifying the negation of (\ref{t:main_new:three}).

We now show ((\ref{t:main_new:two}) $\Rightarrow$ (\ref{t:main_new:one})). Assume that $f$ has zero sums at every $v \in \text{Cu}_2 \setminus \{ \lozenge \}$ without symmetric core. We set $f_v := \sum_{t \in S_v} f(t) \delta_t$ for every $v \in \text{Cu}_2 \setminus \{\lozenge\}$ without symmetric core. As $\text{Cu}_2 \setminus \{\lozenge\}$ is countable, the set of elements without symmetric core may be enumerated as $(v_n)$. In view of Lemma~\ref{l: partition} the set $\{ S_{v_n} \colon n \in \mathbb{N} \}$ consists of mutually disjoint sets, consequently
	\begin{align}\label{eq: aux1}
	\left\| f - \sum\limits_{n=1}^N f_{v_n} \right\| = \sum\limits_{\substack{ t \in \text{Cu}_2 \setminus \{\lozenge\} \\ t \notin \cup_{n=1}^N S_{v_n} }} |f(t)| \rightarrow 0 \qquad (N \rightarrow \infty).
	\end{align}
	The convergence of the right-hand side of \eqref{eq: aux1} follows from Lemma~\ref{l: partition}; namely, that $\{ S_{v_n} \colon n \in \mathbb{N} \}$ covers $\text{Cu}_2 \setminus \{\lozenge\}$. This shows $f \in \overline{\spanning} \{ f_{v_n} \colon n \in \mathbb{N} \}$. Proposition~\ref{p: main2} yields however $f_{v_n} \in \mc J$ for each $n \in \mathbb{N}$. Thus $f \in \mc J$. 

Finally, we show that ((\ref{t:main_new:one}) $\Rightarrow$ (\ref{t:main_new:three})).  Suppose $f \in \mc J$. Assume towards a contradiction that there are $g,h \in \mc A$ with $g \# f \# h = \delta_e$, then $\delta_e \in \mc J$. This is impossible as $\mc A / \mc J$ is non-trivial (in fact infinite-dimensional) by Proposition~\ref{p: rep_quot}.
	
	The equivalence ((\ref{t:main_new:one}) $\Leftrightarrow$ (\ref{t:main_new:three})) shows $\mc J = \mc{M}_{\mc A}$, hence the last part of the theorem follows from Proposition~\ref{p: uniquemax}.
\end{proof}

\begin{corollary}\label{c: almostpi}
	Let $a \in \mc A / \mc J$ be non-zero. Then there exist $b,c \in \mc A / \mc J$ such that $bac =1_{\mc A / \mc J}$.
\end{corollary}

\begin{proof}
	Let $f \in \mc A$ be such that $a= \pi_{\mc J}(f)$. That $a$ is non-zero is equivalent to $f \notin \mc J$. Hence by Theorem~\ref{t: main_new} there are $g,h \in \mc A$ such that $g \# f \# h = \delta_e$. Setting $b:= \pi_{\mc J}(g)$ and $c:= \pi_{\mc J}(h)$ finishes the proof.
\end{proof}

\begin{theorem}\label{t: quot_pi}
	$\mc A / \mc J$ is an infinite-dimensional, purely infinite Banach $*$-algebra.
\end{theorem}

\begin{proof}
	This is immediate from Corollary~\ref{c: almostpi} and Proposition~\ref{p: rep_quot}.
\end{proof}

\subsubsection{A description of the annihilator $\mc J^{\perp}$}

Let us start by pushing the characterisation of $\mc J$ given by \eqref{eq: elements_of_j_1} a bit further:

\begin{lemma}\label{l: alternative_j}
The following holds:
	\begin{align}
	\mc J = \overline{\spanning} \{\delta_{s_{\mathbf{i}}} \# f_0 \# \delta_{s_{\mathbf{j}}^*} \colon \mathbf{i}, \mathbf{j} \in \mathbf{I} \}.
	\end{align}
\end{lemma}

\begin{proof}
	Let us fix $\mathbf{k} \in \mathbf{I} \setminus \mathbf{I}_0$. In view of Lemma~\ref{l: nuts_and_bolts} (1) we have either 
	\begin{itemize}
		\item $s_{\mathbf{k}}^* s_1= \lozenge$ and $s_{\mathbf{k}}^* s_2= s_{\mathbf{p}}^*$, where $\mathbf{p} \in \mathbf{I}$ is such that $\mathbf{k}=2 \mathbf{p}$; or
		\item $s_{\mathbf{k}}^* s_2= \lozenge$ and $s_{\mathbf{k}}^* s_1= s_{\mathbf{q}}^*$, where $\mathbf{q} \in \mathbf{I}$ is such that $\mathbf{k}=1 \mathbf{q}$.
	\end{itemize}
	We may assume without loss of generality that the first item holds. Consequently
	\begin{align}\label{e: alternative_j_calc}
	\delta_{s_{\mathbf{k}}^*} \# f_0 &= \delta_{s_{\mathbf{k}}^*} - \delta_{s_{\mathbf{k}}^*} \# \delta_{s_1 s_1^*} - \delta_{s_{\mathbf{k}}^*} \# \delta_{s_2 s_2^*}
	= \delta_{s_{\mathbf{k}}^*} - 0 - \delta_{s_{\mathbf{p}}^* s_2^*} = 0.
	\end{align}
	With an entirely analogous argument we can show $f_0 \# \delta_{s_{\mathbf{l}}}=0$ for any $\mathbf{l} \in \mathbf{I} \setminus \mathbf{I}_0$.
	
	Hence from \eqref{eq: elements_of_j_1} and the above we conclude
	\begin{align}
	\mc J &= \overline{\spanning} \{\delta_{s_{\mathbf{i}} s_{\mathbf{k}}^*} \# f_0 \# \delta_{s_{\mathbf{l}} s_{\mathbf{j}}^*} \colon \mathbf{i}, \mathbf{j}, \mathbf{k}, \mathbf{l} \in \mathbf{I} \} = \overline{\spanning} \{\delta_{s_{\mathbf{i}}} \# f_0 \# \delta_{s_{\mathbf{j}}^*} \colon \mathbf{i}, \mathbf{j} \in \mathbf{I} \},
	\end{align}
	as required.
\end{proof}

Let us define the maps
\begin{align}
\tau_k \colon \text{Cu}_2 \setminus \{\lozenge\} \to \text{Cu}_2 \setminus \{\lozenge\}; \quad s_{\mathbf{i}} s_{\mathbf{j}}^* \mapsto s_{\mathbf{i}k} s_{\mathbf{j}k}^* \qquad (k \in \{1,2\}).
\end{align}
Lemma~\ref{l: nuts_and_bolts} (2) ensures that $\tau_k$ is in fact well-defined. By the very same result we actually find that $\tau_k$ is injective.

For both $k \in \{1,2\}$, we can find ``induced'' bounded linear operators
\begin{align}\label{eq: t_k}
T_k \colon \mc A \to \mc A \quad \text{with} \quad T_k(\delta_t) = \delta_{\tau_k(t)} \qquad (t \in \text{Cu}_2 \setminus \{\lozenge\}).
\end{align}
From injectivity of $\tau_k$ it easily follows that $T_k$ is an isometry. Let $T:= T_1 + T_2$. Then $T$ is a bounded linear operator on $\mc A$.

In the following, $\mc A^*$ denotes the (continuous) dual of $\mc A$, which we identify with $\ell^{\infty}(\text{Cu}_2 \setminus \{\lozenge\})$ as a Banach space. Let $T^* \colon \mc A^* \to \mc A^*$ denote the adjoint of $T$. We observe that
\begin{align}\label{eq: t_star_action}
(T^* \mu)(s_{\mathbf{i}} s_{\mathbf{j}}^*) &= \langle \delta_{s_{\mathbf{i}} s_{\mathbf{j}}^*}, T^* \mu \rangle = \langle T \delta_{s_{\mathbf{i}} s_{\mathbf{j}}^*}, \mu \rangle = \langle T_1 \delta_{s_{\mathbf{i}} s_{\mathbf{j}}^*}, \mu \rangle + \langle T_2 \delta_{s_{\mathbf{i}} s_{\mathbf{j}}^*}, \mu \rangle \notag \\
&= \langle \delta_{ \tau_1(s_{\mathbf{i}} s_{\mathbf{j}}^*)}, \mu \rangle + \langle \delta_{ \tau_2(s_{\mathbf{i}} s_{\mathbf{j}}^*)}, \mu \rangle = \langle \delta_{s_{\mathbf{i}1} s_{\mathbf{j}1}^*}, \mu \rangle + \langle \delta_{s_{\mathbf{i}2} s_{\mathbf{j}2}^*}, \mu \rangle \notag \\
&= \mu (s_{\mathbf{i}1} s_{\mathbf{j}1}^*) + \mu (s_{\mathbf{i}2} s_{\mathbf{j}2}^*) \qquad (\mathbf{i}, \mathbf{j} \in \mathbf{I}).
\end{align}
We recall that the \textit{annihilator} of $\mc J$ is $\mc J^{\perp}:= \{\mu \in \mc A^* \colon \langle f, \mu \rangle =0 \text{ for all } f \in \mc J \}$. In what follows, $I_{\mc A}$ denotes the identity operator on $\mc A$.

\begin{lemma}\label{l: props_j_perp}
    The following hold:
	\begin{enumerate}
		\item[(1)] $\mc J  = \overline{\Ran}(I_{\mc A} - T)$, and
		\item[(2)] $\mc J^{\perp} = \{\mu \in \mc A^* \colon T^* \mu = \mu \}$.
	\end{enumerate}
\end{lemma}

\begin{proof}
	That $\mc J  =\overline{\Ran}(I_{\mc A} - T)$ is immediate from Lemma~\ref{l: alternative_j} and \eqref{eq: t_k}.
	
	We now prove (2). Let us fix $\mu \in \mc A^*$. Suppose first $T^* \mu = \mu$. Then $\langle f, \mu \rangle = \langle f, T^* \mu \rangle = \langle Tf, \mu \rangle$ or equivalently $\langle f- Tf, \mu \rangle =0$ for every $f \in \mc A$. Hence by continuity, $\langle g, \mu \rangle =0$ for every $g \in \overline{\Ran}(I_{\mc A} - T)$. By (1) this is equivalent to $\mu \in \mc J^{\perp}$.
	
	In the other direction suppose $\mu \in \mc J^{\perp}$. By (1) we clearly have $f-Tf \in \overline{\Ran}(I_{\mc A} - T) = \mc J$, and hence $\langle f -Tf, \mu \rangle=0$ or equivalently $\langle f, \mu \rangle= \langle Tf, \mu \rangle = \langle f, T^*\mu \rangle$ for each $f \in \mc A$. Thus $T^* \mu = \mu$.
\end{proof}

\subsubsection{$\mc A / \mc J$ does not have purely infinite ultrapowers}

\begin{proposition}\label{p: pre_not_bdd_below}
	Let $F \subseteq \mathbf{I}$ be finite with $\emptyset \notin F$, and set $f:= \sum_{\mathbf{i} \in F} \delta_{s_{\mathbf{i}}^*}$. Then $\| \pi_{\mc J}(f) \| = |F|$.
\end{proposition}

\begin{proof}
	Clearly $\|f\| = \sum_{\mathbf{i} \in F} \| \delta_{s_{\mathbf{i}}^*} \| = |F|$ and hence $\| \pi_{\mc J}(f) \| \leqslant |F|$. Thus it suffices to show $\| \pi_{\mc J}(f) \| \geqslant |F|$. This in turn follows if we can find $\xi \in (\mc A / \mc J)^*$ satisfying $\| \xi \| =1$ and $| \langle \pi_{\mc J} (f), \xi \rangle | \geqslant | F |$. Recall, by Hahn--Banach, that $\pi_{\mc J}^* \colon (\mc A / \mc J)^* \to \mc A^*$ is a linear isometry with range equal to $\mc J^{\perp}$. Hence it is sufficient to find $\mu \in \mc J^{\perp}$ satisfying $\| \mu \| =1$ and $| \langle f, \mu \rangle | \geqslant |F|$.
	
	We shall now define such a $\mu$. To this end, let us consider the following property. Given $\alpha \in \mathbb{N}_0$ we say that $t \in \text{Cu}_2 \setminus \{\lozenge\}$ has property $(\alpha -\maltese)$ if
	\begin{align*}
	t=s_{\mathbf{i}} s_{\mathbf{kj}}^* \text{ for some } \mathbf{i}, \mathbf{j} \in \mathbf{I}_{\alpha}, \text{ and } \mathbf{k} \in F. \qquad (\alpha - \maltese)
	\end{align*}
	Now define $\mu \colon \text{Cu}_2 \setminus \{\lozenge\} \to \mathbb{C}$ by setting
	\begin{align}\label{eq:mu_defn}
	\mu(t):= \begin{cases}
	2^{- \alpha} &\text{if } t \text{ has property  } (\alpha-\maltese) \text{ for some } \alpha \in \mathbb{N}_0 \\
	0 &\text{otherwise}
	\end{cases} \quad (t \in \text{Cu}_2 \setminus \{\lozenge\}).
	\end{align}
	We need to check that $\mu$ is well-defined. Assume $\alpha, \beta \in \mathbb{N}_0$, $\mathbf{i}, \mathbf{j} \in \mathbf{I}_{\alpha}$, $\mathbf{p}, \mathbf{q} \in \mathbf{I}_{\beta}$ and $\mathbf{k}, \mathbf{l} \in F$ are such that $s_{\mathbf{i}} s_{\mathbf{kj}}^* = s_{\mathbf{p}} s_{\mathbf{lq}}^*$. Then it follows from Lemma~\ref{l: nuts_and_bolts} (2) that $\mathbf{i}= \mathbf{p}$ and hence $\alpha = \beta$.
	
	It is clear that $\mu$ is bounded with $\| \mu \| =1$, hence $\mu \in \mc A^*$. We want to show that in fact $\mu \in \mc J^{\perp}$, which in view of Lemma~\ref{l: props_j_perp} (2) is equivalent to the following claim.
	
	\begin{claim}
		$\mu = T^* \mu$.
	\end{claim}
	
	\begin{proof}[Proof of Claim.]
		\textit{Assume first $t \in \text{Cu}_2 \setminus \{\lozenge\}$ has property $(\alpha- \maltese)$ for some $\alpha \in \mathbb{N}_0$.} Then $t=s_{\mathbf{i}} s_{\mathbf{kj}}^*$ for some $\mathbf{i}, \mathbf{j} \in \mathbf{I}_{\alpha}$ and $\mathbf{k} \in F$. Notice that $s_{\mathbf{i}1} s_{\mathbf{kj}1}^*$ and $s_{\mathbf{i}2} s_{\mathbf{kj}2}^*$ have property $((\alpha +1) - \maltese)$, hence by \eqref{eq: t_star_action}
		\begin{align}\label{eq:translations_mu}
		\mu(t) &= 2^{- \alpha} = 2^{- \alpha-1} + 2^{- \alpha-1} = \mu (s_{\mathbf{i}1} s_{\mathbf{kj}1}^*) + \mu (s_{\mathbf{i}2} s_{\mathbf{kj}2}^*) \notag \\
		&= (T^* \mu)(s_{\mathbf{i}} s_{\mathbf{kj}}^*) = (T^* \mu)(t).
		\end{align}
		\textit{Assume now $t \in \text{Cu}_2 \setminus \{\lozenge\}$ does not have property $(\alpha - \maltese)$ for any $\alpha \in \mathbb{N}_0$.}	By definition, $\mu(t) =0$. By Lemma~\ref{l: nuts_and_bolts} (2) we can find unique $\alpha, \beta \in \mathbb{N}_0$ and $\mathbf{i} \in \mathbf{I}_{\alpha}$, $\mathbf{j} \in \mathbf{I}_{\beta}$ such that $t=s_{\mathbf{i}} s_{\mathbf{j}}^*$.
		
		We observe that $s_{\mathbf{i}1} s_{\mathbf{j}1}^*$ does not have property $(\gamma - \maltese)$ for any $\gamma \in \mathbb{N}_0$. For assume towards a contradiction that there exist $\gamma \in \mathbb{N}_0$, $\mathbf{p}, \mathbf{q} \in \mathbf{I}_{\gamma}$ and $\mathbf{l} \in F$ such that $s_{\mathbf{i}1} s_{\mathbf{j}1}^* = s_{\mathbf{p}} s_{\mathbf{lq}}^*$. Then $\mathbf{i}1 = \mathbf{p}$ and $\mathbf{j}1 = \mathbf{lq}$. In particular $\alpha +1 = \gamma$ and $\mathbf{lq} \in \mathbf{I}_{\beta +1}$.
		\begin{itemize}
			\item \textit{Suppose $\alpha \geqslant \beta$.} By the above $\mathbf{l} \in \mathbf{I}_{\beta +1 - \gamma} = \mathbf{I}_{\beta- \alpha}$.  Consequently $\mathbf{l} \in \mathbf{I}_0$ must hold, which is equivalent to saying $\mathbf{l}= \emptyset$. This contradicts $\emptyset \notin F$.
			
			\item \textit{Suppose $\alpha < \beta$.} Then $\mathbf{j}= \mathbf{wu}$ for some $\mathbf{u} \in \mathbf{I}_{\alpha}$ and $\mathbf{w} \in \mathbf{I}_{\beta - \alpha}$ and therefore $t= s_{\mathbf{i}} s_{\mathbf{wu}}^*$. As $t$ does not have property $(\alpha - \maltese)$ it follows that $\mathbf{w} \notin F$. However $\mathbf{u}1 \in \mathbf{I}_{\alpha +1}$ and $\mathbf{q} \in \mathbf{I}_{\alpha +1}$, thus from $\mathbf{wu}1= \mathbf{j}1= \mathbf{lq}$ we conclude $\mathbf{w}= \mathbf{l} \in F$, a contradiction.
		\end{itemize}
		An analogous argument shows that $s_{\mathbf{i}2} s_{\mathbf{j}2}^*$ does not have property $(\gamma - \maltese)$ either for any $\gamma \in \mathbb{N}_0$. From \eqref{eq: t_star_action} we obtain $(T^* \mu)(t) = (T^* \mu)(s_{\mathbf{i}} s_{\mathbf{j}}^*) = \mu(s_{\mathbf{i}1} s_{\mathbf{j}1}^*) + \mu(s_{\mathbf{i}2} s_{\mathbf{j}2}^*) =0$, as required.
	\end{proof}
	
	Lastly, from the definition of $\mu$ we see
	\begin{align}
	|\langle f, \mu \rangle| = \Big| \sum_{\mathbf{i} \in F} \langle \delta_{s_{\mathbf{i}}^*}, \mu \rangle \Big| =\Big| \sum_{\mathbf{i} \in F} \mu(s_{\mathbf{i}}^*) \Big| = \sum_{\mathbf{i} \in F} 2^{-0} = |F|.
	\end{align}
	Hence the proposition is proved.
\end{proof}

\begin{proposition}\label{p: not_bdd_below}
	Fix $p \in [1, \infty)$. Let $\Theta \colon \mc A/\mc J \rightarrow \mc B(\ell^p)$ be a continuous algebra homomorphism such that $\Theta(\pi_{\mc J}(\delta_{s_i^*}))=A_i$ for all $i \in \{1,2 \}$. Then $\Theta$ is injective but it is not bounded below.
\end{proposition}

\begin{proof}
	Let us consider the operator $S:= (A_1 + A_2)/2 \in \mc B(\ell^p)$. We \textit{claim} that $\|S\| = 2^{-1/p}$. To see this, we first observe that $(Sx)(k) = (x_{2k} + x_{2k-1})/2$ for all $x \in \ell^p$ and $k \in \mathbb{N}$. Thus for each $x \in \ell^p$,
	\begin{align}\label{e: norm_s}
	\|Sx\|^p \leqslant \sum_{k=1}^{\infty} \Big( \dfrac{|x_{2k-1}| + | x_{2k} |}{2} \Big)^p \leqslant \sum_{k=1}^{\infty} \Big( \dfrac{|x_{2k-1}|^p + | x_{2k} |^p}{2} \Big) = \dfrac{1}{2} \sum_{n=1}^{\infty} |x_n|^p = \dfrac{1}{2} \|x\|^p .
	\end{align}
	The second inequality in \eqref{e: norm_s} follows from convexity of the function $x \mapsto x^p; \; [0, \infty) \to [0, \infty)$. (Specifically, we use that $(a+b)^p/2^p \leqslant (a^p + b^p)/2$ whenever $a, b \geqslant 0$.) Thus $\|S\| \leqslant 2^{-1/p}$. The upper bound is sharp, as $\|Sx\| = 2^{-1/p}\|x\|$ holds by choosing for example $x:=e_1+e_2$.
	
	We set $h:= (\delta_{s_1^*} + \delta_{s_2^*})/2$. For any $N \in \mathbb{N}$ we see that $h^N = 2^{-N}\sum_{\mathbf{i} \in \mathbf{I}_N} \delta _{s_{\mathbf{i}}^*}$, where clearly $| \mathbf{I}_N | = 2^N$ and $\emptyset \notin \mathbf{I}_N$. Therefore by Proposition~\ref{p: pre_not_bdd_below} we obtain
	\begin{align}\label{e: not_bdd_below1}
	\left\| \pi_{\mc J} \left( h^N \right) \right\| = 2^{-N} \Big\| \pi_{\mc J} \Big( \sum_{\mathbf{i} \in \mathbf{I}_N} \delta _{s_{\mathbf{i}}^*} \Big) \Big\| = 2^{-N} |\mathbf{I}_N | = 1 \qquad (N \in \mathbb{N}).
	\end{align}
	From $\Theta (\pi_{\mc J}(h)) = S$ and that $\Theta$ and $\pi_{\mc J}$ are homomorphisms we obtain	
	\begin{align}\label{e: not_bdd_below2}
	\left\| \Theta \left( \pi_{\mc J} \left( h^N \right) \right) \right\| \leqslant
	\left\| \Theta \left( \pi_{\mc J} \left( h \right) \right) \right\|^N =
	\|S\|^N = 2^{-N/p} \qquad (N \in \mathbb{N}). 
	\end{align}
	It follows from \eqref{e: not_bdd_below1} and \eqref{e: not_bdd_below2} that $\Theta$ cannot be bounded below.
	
	Lastly, $\mc A / \mc J$ is purely infinite by Theorem~\ref{t: quot_pi}, hence in particular it is simple by Lemma~\ref{l: pui}~(1). As $\Theta$ is a non-zero continuous algebra homomorphism, $\Ker(\Theta) = \{0\}$ must hold.
\end{proof}

\begin{remark}
We stated this result for an arbitrary homomorphism $\Theta$ with $\Theta(\pi_{\mc J}(\delta_{s_i^*}))=A_i$ for $i\in\{1,2\}$, as this is all the argument needed.  In fact, such a $\Theta$ is already equal to $\Theta_p$, as defined above.  Indeed, let $C_i = \Theta(\pi_{\mc J}(\delta_{s_i}))$ for $i\in\{1,2\}$. Then we have that $A_1 C_1= I_{\ell^p}= A_2 C_2$ and $A_1 C_2= 0 =A_2 C_1$ and $C_1A_1+C_2A_2 = I_{\ell^p}$. Thus $B_1 = B_1 A_1 C_1 = (I_{\ell^p} - B_2 A_2) C_1 = C_1 - B_2 A_2 C_1 = C_1$ and symmetrically $B_2 = C_2$.
From this, \eqref{eq: elements_of_a} and Proposition~\ref{p: rep_quot} the claim readily follows.
\end{remark}

\begin{theorem}\label{t: ultra_not_pi}
	The Banach $*$-algebra $(\mc A / \mc J)_{\mc U}$ is not simple hence not purely infinite for any countably-incomplete ultrafilter $\mc U$.
\end{theorem}

\begin{proof}
	Taking for example $p=1$, it follows from Proposition~\ref{p: rep_quot} that there is a continuous, unital algebra homomorphism $\Theta \colon \mathcal{A} / \mathcal{J} \rightarrow \mathcal{B}(\ell^1)$ with $\Theta(\pi_{\mc J}(\delta_{s_i^*}))=A_i$ for each $i \in \{1,2 \}$. Thus $\Theta$ is not bounded below by Proposition~\ref{p: not_bdd_below}. Hence $(\mc A / \mc J)_{\mc U}$ cannot be simple for any countably-incomplete ultrafilter $\mc U$ by Proposition~\ref{p: pi_implies_bdd_below}. In particular, $(\mc A / \mc J)_{\mc U}$ is not purely infinite by Lemma~\ref{l: pui} (1).
\end{proof}

\begin{remark}
	Even though $(\mc A / \mc J)_{\mc U}$ is not purely infinite for any countably-incomplete ultrafilter $\mc U$, it is always properly infinite. Indeed, $\mc A / \mc J$ is purely infinite by Theorem~\ref{t: quot_pi} hence it is properly infinite by Lemma~\ref{l: pui} (2). Now it follows from \cite[Corollary~4.19]{dh} that an ultrapower of a properly infinite Banach algebra is properly infinite, hence $(\mc A / \mc J)_{\mc U}$ is properly infinite for any ultrafilter $\mc U$.
\end{remark}

\begin{remark}\label{rem:phillips}
In \cite{p1} Phillips considers certain representations of the Leavitt algebra $L_2$ (see Remark~\ref{rem:leavitt_algs}) on $L^p$ spaces. Indeed, \cite[Example~3.1]{p1} constructs a representation of $L_2$ on $\ell^p$ which is essentially the same as the restriction of our $\Theta_p$ to $L_2$. Phillips explores generalisations of these representations, which are called \emph{spatial}, see \cite[Definition~7.4, Lemma~7.5]{p1}.  It is shown in \cite[Theorem~8.7]{p1} that all spatial representations give rise to isometrically isomorphic closures.  This gives rise to the $p$-analogues of the Cuntz algebras, \cite[Definition~8.8]{p1}; see also \cite{cgt} for more recent study of these algebras.
Thus the closure of the image of $\Theta_p$, inside $\mc B(\ell^p)$, is isometric to $\mc O^p_2$, in the language of \cite{p1}.  Our result of course shows that $\Theta_p |^{\mc O^p_2} \colon \mc A/\mc J \to \mc O^p_2$ is \emph{not} an isomorphism, because it is not bounded below. In fact, as we shall see in Section~\ref{s: traces}, the Banach algebras $\mc A / \mc J$ and $\mc O^p_2$ are \textit{not} isomorphic for any $p \in [1, \infty)$.

Phillips shows in \cite{p2} that, in particular, $\mc O^p_2$ is purely infinite (with the same definition as we use).  
The proof, however, is different to our proof that $\mc A/\mc J$ is purely infinite, and much more closely parallels the $C^*$-algebraic proof that $\mc O_2$ is purely infinite.  A close examination of the proof shows that it does not work for $\mc A/\mc J$, as various necessary norm estimates are different (in the sense of not even being equivalent up to a constant) for $\mc A/\mc J$.

It is not obvious to us that the proof in \cite{p2} provides an estimate for how $C_{\text{pi}}^{\mc O^p_2}$ behaves, and hence if $\mc O^p_2$ has purely infinite ultrapowers. Furthermore, given a lack of nice ``permanence'' properties for purely infinite Banach algebras, it seems that knowing $\mc O^p_2$ is purely infinite is no direct help in showing that $\mc A/\mc J$ is purely infinite, or vice versa. We remark that similar questions around
``permanence properties'' are raised at the end of \cite{agps}.
\end{remark}

\begin{remark}
As a final remark, while we have chosen to work with $\text{Cu}_2$ in this paper (or, equivalently, with an $\ell^1$-completion of the Leavitt algebra $L_2$, see Remark~\ref{rem:leavitt_algs}) one could also follow Phillips, and consider the general $\text{Cu}_d$ and $L_d$, for $d\geqslant 2$.  We have chosen not to do this for notational simplicity, but let us quickly indicate what changes would be needed.

$\text{Cu}_d$ has generators $(s_i)_{i=1}^d, (s_i^*)_{i=1}^d$ with $s_i^*s_i=e$ and $s_i^*s_j=\lozenge$ for $i\not=j$.  The combinatorics of $\text{Cu}_d$ are essentially the same, just with $\bf I$ now being all finite sequences in $\{1,2,\cdots,d\}$.  We can then form $\mc A_d := \ell^1(\text{Cu}_d\setminus\{\lozenge\},\#)$ and consider the ideal $\mc J_d$ generated by the element $\delta_e - \sum_{i=1}^d \delta_{s_is_i^*}$.  All of the results continue to hold, with essentially identical proofs, excepting that various statements in the proof of Proposition~\ref{p: main2} now need to sum over $1,2,\cdots,d$ instead of just $1,2$.

To represent $\mc A_d/\mc J_d$ on $\ell^p$, we simply replace ``$2n$'' by ``$dn$'', for example, defining
\begin{align}
 (A_kx)(n) = x_{dn+1-k} \qquad \left( x\in\ell^p, \; k \in \{1,2, \ldots, d \} \right).
\end{align}
The obvious modifications can be made to Proposition~\ref{p: rep_quot}.  Similarly, $d$ cases, instead of just two, need to be considered in the proof of Lemma~\ref{l: alternative_j}, and in the definitions of $\tau_k$ after this (now $k \in \{1,2,\cdots,d\}$), and $2$ replaced by $d$ in (\ref{eq: t_star_action}), (\ref{eq:mu_defn}) and (\ref{eq:translations_mu}).  Finally, we perform similar alterations to Proposition~\ref{p: not_bdd_below}.
\end{remark}

\section{Bounded traces on $\mathcal{A}/ \mathcal{J}$ and the lack thereof on $\mathcal{O}_2^p$}\label{s: traces}	

The main aim of this section is to prove that there exist non-zero bounded traces on $\mathcal{A}/ \mathcal{J}$ (Theorem~\ref{t: trace_on_aj}), while there are no non-zero bounded traces on $\mathcal{O}_2^p$ (Theorem~\ref{t: no_trace_on_o_2}). These two results immediately imply the following in particular:

\begin{theorem}\label{t: aj_o_2_nonisomorphic}
	The Banach algebras $\mc A / \mc J$ and $\mathcal{O}_2^p$ are not isomorphic for any $p \in [1, \infty)$.
\end{theorem}

\subsection{Some background on traces}
Let us first recall some basic properties of traces. A \textit{trace} on a Banach algebra $\mathcal{B}$ is a linear functional $\tau$ on $\mathcal{B}$ with $\langle ab, \tau \rangle = \langle ba, \tau \rangle$ for each $a,b \in \mathcal{B}$. We say that a trace $\tau$ on $\mathcal{B}$ is bounded if $\tau \in \mathcal{B}^*$. The \textit{trace space} of $\mathcal{B}$ is the set of all bounded traces on $\mathcal{B}$. A trace $\tau$ on a unital Banach algebra $\mathcal{B}$ is \textit{normalised} if $\langle 1_{\mathcal{B}}, \tau \rangle = 1$.

Let $\mathcal{B}$ be a unital Banach algebra and let $\mathcal{I}$ be a closed, two-sided ideal in $\mathcal{B}$.  If $\tau$ is a bounded trace on $\mathcal B$ with $\mathcal I \subseteq \Ker(\tau)$, then there is a unique bounded linear map $\tau':\mathcal B/\mathcal I\rightarrow\mathbb C$ with $\tau'\circ\pi_{\mathcal I}=\tau$, and $\tau'$ is readily seen to be a trace, non-zero if and only if $\tau$ is non-zero.  We say that $\tau$ \emph{drops to a trace on $\mathcal B/\mathcal I$}.  Any bounded trace $\tau'$ on $\mathcal B/\mathcal I$ arises from such a $\tau$.  Hence the trace space of $\mathcal{B} / \mathcal{I}$ is in bijection with the set of bounded traces $\tau$ on $\mathcal{B}$ for which $\mathcal{I} \subseteq \Ker(\tau)$ holds.

\subsection{There \textit{are} non-zero bounded traces on $\mathcal{A}/ \mathcal{J}$}

Recall that $\ell^{1}(\text{Cu}_2)^*$ is isometrically isomorphic to $\ell^{\infty}(\text{Cu}_2)$ and that $\mathcal{A} := (\ell^{1}(\text{Cu}_2 \setminus \{ \lozenge \}), \#)$ is isometrically isomorphic to $\ell^{1}(\text{Cu}_2) / \mathbb{C} \delta_{\lozenge}$. Hence by the previous section we may identify the trace space of $\mathcal{A}$ with the set
\begin{align}\label{e: traces_on_A}
\left\{ \tau |_{\text{Cu}_2 \setminus \{ \lozenge \}} \colon \tau \in \ell^{\infty}(\text{Cu}_2), \; \tau(\lozenge) = 0, \; \tau(uv) = \tau(vu) \quad (u,v \in \text{Cu}_2) \right\}.
\end{align}

Recall that $(\mathcal{A} / \mathcal{J})^*$ and $\mathcal{J}^{\perp}$ are isometrically isomorphic, and that $\mathcal{J}^{\perp}$ embeds isometrically into $\ell^{\infty}(\text{Cu}_2 \setminus \{\lozenge\})$. Henceforth we identify traces on $\mathcal{A} / \mathcal{J}$ with a subset of $\ell^{\infty}(\text{Cu}_2 \setminus \{\lozenge\})$.

\begin{lemma}\label{l: char_trace}
	Let $\tau$ be a bounded trace on $\mathcal{A}$. Then $\mathcal{J} \subseteq \Ker(\tau)$ if and only if $\tau(e) = 0$. Consequently, $\tau$ drops to a bounded trace on $\mathcal{A} / \mathcal{J}$ if and only if $\tau(e) =0$.
\end{lemma}

\begin{proof}
	Let us first note that by the definition of $f_0$ it follows that
	\begin{align}\label{e: trace_at_f0}
	\langle f_0, \tau \rangle &= \langle \delta_e, \tau \rangle - \langle \delta_{s_1 s_1^*}, \tau \rangle  - \langle \delta_{s_2 s_2^*}, \tau \rangle = \tau(e) - \tau(s_1 s_1^*) - \tau(s_2 s_2^*) \notag \\
	&= \tau(e) - \tau(s_1^* s_1) - \tau(s_2^* s_2) = \tau(e) - 2 \tau(e) = -\tau(e).
	\end{align}
	\textit{Suppose $\tau(e)=0$ holds.} We show that $\mathcal J \subseteq \Ker(\tau)$.  By Lemma~\ref{l: alternative_j} it is enough to show that $\langle \delta_{s_{\mathbf{i}}} \# f_0 \# \delta_{s_{\mathbf{j}}^*}, \tau \rangle =0$ for $\mathbf{i}, \mathbf{j} \in \mathbf{I}$.   We consider a number of possibilities.  Suppose there is a $\mathbf{k} \in \mathbf{I}$ such that $\mathbf{j} = \mathbf{ik}$, then by Lemma~\ref{l: nuts_and_bolts} we have $s_{\mathbf{j}}^* s_{\mathbf{i}} = s_{\mathbf{k}}^*$. Thus
	\begin{align}
	\langle \delta_{s_{\mathbf{i}}} \# f_0 \# \delta_{s_{\mathbf{j}}^*}, \tau \rangle = \langle \delta_{s_{\mathbf{j}}^*} \# \delta_{s_{\mathbf{i}}} \# f_0, \tau \rangle = \langle \delta_{s_{\mathbf{k}}^*} \# f_0, \tau \rangle.
	\end{align}
	If $\mathbf{k} \neq \emptyset$, then from \eqref{e: alternative_j_calc} we see that $\delta_{s_{\mathbf{k}}^*} \# f_0 =0$, and so $\langle \delta_{s_{\mathbf{k}}^*} \# f_0, \tau \rangle = 0$.	If $\mathbf{k} = \emptyset$ then $s_{\mathbf{k}}^* = e$ and thus $\delta_{s_{\mathbf{k}}^*} \# f_0 = f_0$. Hence $\langle \delta_{s_{\mathbf{k}}^*} \# f_0, \tau \rangle = \langle f_0, \tau \rangle = - \tau(e) =0$ by \eqref{e: trace_at_f0} and the assumption $\tau(e) =0$. 
	
	Suppose there is a $\mathbf{k} \in \mathbf{I}$ such that $\mathbf{i} = \mathbf{jk}$, then by Lemma~\ref{l: nuts_and_bolts} we have $s_{\mathbf{j}}^* s_{\mathbf{i}} = s_{\mathbf{k}}$. Thus $\langle \delta_{s_{\mathbf{i}}} \# f_0 \# \delta_{s_{\mathbf{j}}^*}, \tau \rangle = \langle f_0 \#  \delta_{s_{\mathbf{j}}^*} \# \delta_{s_{\mathbf{i}}}, \tau \rangle = \langle f_0 \# \delta_{s_{\mathbf{k}}}, \tau \rangle$. An analogous reasoning to the above shows that $\langle f_0 \# \delta_{s_{\mathbf{k}}}, \tau \rangle =0$.
	
	Finally, if there is no $\mathbf{k} \in \mathbf{I}$ such that $\mathbf{i} = \mathbf{jk}$ or $\mathbf{j} = \mathbf{ik}$, then $s_{\mathbf{j}}^* s_{\mathbf{i}} = \lozenge$ by Lemma~\ref{l: nuts_and_bolts} and hence $\delta_{s_{\mathbf{j}}^*} \# \delta_{s_{\mathbf{i}}} =0$. Thus $\langle \delta_{s_{\mathbf{i}}} \# f_0 \# \delta_{s_{\mathbf{j}}^*}, \tau \rangle = \langle \delta_{s_{\mathbf{j}}^*} \# \delta_{s_{\mathbf{i}}} \# f_0, \tau \rangle =0$. \smallskip
	
	\textit{Suppose $\mathcal{J} \subseteq \Ker(\tau)$ holds.} Then $0 = \langle f_0, \tau \rangle = - \tau(e)$ by \eqref{e: trace_at_f0}, hence $\tau(e) =0$ as claimed.
\end{proof}

There is a complete characterisation of traces on $\mathcal{A}$ given in \cite[Corollary~3.13]{dlr}, which could presumably be extended to a complete characterisation of traces on $\mathcal{A}/\mathcal{J}$.  We restrict our study here to the bare minimum needed for our purposes.

\begin{theorem}\label{t: trace_on_aj}
	Consider the map $\tau \colon \text{Cu}_2 \rightarrow \mathbb{C}$ defined as
	\begin{align}
	\tau(w) = \begin{cases} 1 &\text{if } w = s_{\mathbf{i}1} s_{\mathbf{i}}^* \text{ for some } \mathbf{i} \in \mathbf{I}, \\ 0 & \text{otherwise.} \end{cases}
	\end{align}
	The restriction of $\tau$ onto $\text{Cu}_2 \setminus \{\lozenge\}$ gives a non-zero bounded trace on $\mathcal{A}$ with $\tau(e)=0$. Consequently, $\tau$ drops to a non-zero bounded trace on $\mathcal{A} / \mathcal{J}$.
\end{theorem}

\begin{proof}
	That $\tau$ is non-zero follows from \textit{e.g.}\ $\tau(s_1) =1$. Also, $\tau(e)=0$ is immediate.  Rather than using the characterisation given by \cite[Corollary~3.13]{dlr}, which itself requires checking many conditions, we shall show directly that our example $\tau$ is indeed a trace.
	
	Take arbitrary $v,u \in \text{Cu}_2 \setminus \{\lozenge\}$ and write $v = s_{\mathbf{i}} s_{\mathbf{j}}^*$ and $u = s_{\mathbf{k}} s_{\mathbf{l}}^*$ for some $\mathbf{i}, \mathbf{j}, \mathbf{k}, \mathbf{l} \in \mathbf{I}$. From Lemma~\ref{l: nuts_and_bolts} we see that
	\begin{align}
	vu = s_{\mathbf{i}} s_{\mathbf{j}}^* s_{\mathbf{k}} s_{\mathbf{l}}^* = \begin{cases} s_{\mathbf{i}} s_{\mathbf{lp}}^* &\text{if } \mathbf{j} = \mathbf{kp} \text{ for some } \mathbf{p} \in \mathbf{I}, \qquad \text{(a)} \\ s_{\mathbf{ip}} s_{\mathbf{l}}^* &\text{if } \mathbf{k} = \mathbf{jp} \text{ for some } \mathbf{p} \in \mathbf{I}, \qquad \text{(b)} \\ \lozenge & \text{otherwise} \qquad \text{(c)} \end{cases}
	\end{align}
	and
	\begin{align}
	uv = s_{\mathbf{k}} s_{\mathbf{l}}^* s_{\mathbf{i}} s_{\mathbf{j}}^* = \begin{cases} s_{\mathbf{k}} s_{\mathbf{jq}}^* &\text{if } \mathbf{l} = \mathbf{iq} \text{ for some } \mathbf{q} \in \mathbf{I}, \qquad \text{(I)} \\ s_{\mathbf{kq}} s_{\mathbf{j}}^* &\text{if } \mathbf{i} = \mathbf{lq} \text{ for some } \mathbf{q} \in \mathbf{I}, \qquad \text{(II)} \\ \lozenge & \text{otherwise.} \qquad \text{(III)} \end{cases}
	\end{align}
	Suppose $vu= s_{\mathbf{n}1} s_{\mathbf{n}}^*$ for some $\mathbf{n} \in \mathbf{I}$. Then $\tau(vu)=1$. Clearly (c) cannot hold.
	\begin{itemize}
		\item If (a) holds then $s_{\mathbf{i}} s_{\mathbf{lp}}^* = s_{\mathbf{n}1} s_{\mathbf{n}}^*$ and hence $\mathbf{lp}=\mathbf{n}$ and $\mathbf{i} = \mathbf{n}1 = \mathbf{lp}1$. This yields that we are in case (II) with $\mathbf{q} = \mathbf{p}1$. Hence $uv = s_{\mathbf{kq}} s_{\mathbf{j}}^* = s_{\mathbf{kp}1} s_{\mathbf{kp}}^*$ and thus $\tau(uv)=1$.
		\item If (b) holds then $s_{\mathbf{ip}} s_{\mathbf{l}}^* = s_{\mathbf{n}1} s_{\mathbf{n}}^*$ and hence $\mathbf{l}=\mathbf{n}$ and $\mathbf{ip} = \mathbf{n}1 = \mathbf{l}1$. Note that either:
		\begin{itemize}
			\item $\mathbf{p} = \emptyset$ so $\mathbf{i} = \mathbf{l}1$. This yields that we are in case (II) with $\mathbf{q} = 1$. Hence $uv = s_{\mathbf{kq}} s_{\mathbf{j}}^* =s_{\mathbf{k}1} s_{\mathbf{j}}^* =s_{\mathbf{jp}1} s_{\mathbf{j}}^* = s_{\mathbf{j}1} s_{\mathbf{j}}^*$ and thus $\tau(uv)=1$; or
			\item $\mathbf{p} = \mathbf{r}1$ for some $\mathbf{r} \in \mathbf{I}$ then $\mathbf{ir}1 = \mathbf{l}1$ and thus $\mathbf{ir} = \mathbf{l}$. Note that this yields that we are in case (I) with $\mathbf{q} = \mathbf{r}$. Hence $uv = s_{\mathbf{k}} s_{\mathbf{jq}}^* = s_{\mathbf{jp}} s_{\mathbf{jq}}^* = s_{\mathbf{jr}1} s_{\mathbf{jr}}^*$ and thus $\tau(uv)=1$.
		\end{itemize}
	\end{itemize}
	Alternatively, $vu \neq s_{\mathbf{m}1} s_{\mathbf{m}}^*$ for any $\mathbf{m} \in \mathbf{I}$. Then $\tau(vu)=0$.  If $\tau(uv)=1$ then $uv= s_{\mathbf{n}1} s_{\mathbf{n}}^*$ for some $\mathbf{n} \in \mathbf{I}$, so by the above argument with $u,v$ swapped, also $vu$ is of this form, a contradiction.  Hence $\tau(uv)=0$.
	
	In either case, $\tau(uv)= \tau(vu)$ follows, and so $\tau$ is a trace.  The ``consequently'' part of the theorem is a corollary of the first part and Lemma~\ref{l: char_trace}.
\end{proof}	

\subsection{There are \textit{no} non-zero bounded traces on $\mathcal{O}_2^p$}

For the rest of the section we fix some $p \in [1, \infty)$. Recall the definitions of the operators $A_1, A_2$ and $B_1, B_2$ on $\ell^p$. If $(e_n)$ denotes the standard unit vector basis on $\ell^p$ then we easily conclude that
\begin{align}\label{e: actionofai}
A_i e_n = \begin{cases} e_{\frac{n+i-1}{2}} &\text{if } n+i-1 \text{ is even}, \\ 0 & \text{otherwise} \end{cases}
\qquad (n \in \mathbb{N}, \, i \in \{1,2\}).
\end{align}
Similarly
\begin{align}\label{e: actionofbi}
B_i e_n =	e_{2n-i+1} \qquad (n \in \mathbb{N}, \, i \in \{1,2\}).
\end{align}
Given $\mathbf{i} \in \mathbf{I}$ we define
\begin{align}
A_{\mathbf{i}}:= \begin{cases}
A_{i_1} A_{i_2} \cdots A_{i_\alpha} &\text{ if } \mathbf{i}= (i_1,i_2, \ldots i_\alpha) \in \mathbf{I} \setminus \{\emptyset\}, \\
I_{\ell^p} &\text{ if } \mathbf{i}= \emptyset.
\end{cases}
\end{align}
Analogously, define $B_{\bf i}$.
\smallskip

Let us introduce a piece of notation. For a fixed $\mathbf{i} \in \mathbf{I}$ we define
\begin{align}
\mathbf{i}^{*}:= \begin{cases}
(i_{\alpha}, i_{\alpha -1}, \ldots i_1) &\text{ if } \mathbf{i}= (i_1,i_2, \ldots i_\alpha) \in \mathbf{I} \setminus \{\emptyset\}, \\
\emptyset &\text{ if } \mathbf{i} = \emptyset.
\end{cases}
\end{align}
Clearly $A_{\mathbf{i}^{*}} B_{\mathbf{i}} = I_{\ell^p}$ for any $\mathbf{i} \in \mathbf{I}$. \medskip

Given $\mathbf{i} \in \mathbf{I}$, we define the map $\rho_{\mathbf{i}} \colon \mathbb{N} \to \mathbb{R}$ by
\begin{align}
\rho_{\mathbf{i}}(n):= \begin{cases}
2^{- \alpha} \left( n+ \sum\limits_{l=1}^{\alpha} 2^{l-1} i_l -2^{\alpha} +1 \right) &\text{ if } \mathbf{i}= (i_1,i_2, \ldots i_{\alpha}) \in \mathbf{I} \setminus \{\emptyset\}, \\
n &\text{ if } \mathbf{i}= \emptyset
\end{cases} \quad (n \in \mathbb{N}).
\end{align}
An elementary induction argument on $\alpha$ and \eqref{e: actionofai} shows that
\begin{align}\label{e: actionofaistar}
A_{\mathbf{i}^{*}} e_n = \begin{cases} e_{\rho_{\mathbf{i}}(n)} &\text{if } n+ \sum\limits_{l=1}^{\alpha} 2^{l-1} i_l -2^{\alpha} +1 \in 2^{\alpha} \mathbb{N}, \\ 0 & \text{otherwise} \end{cases}
\qquad (n \in \mathbb{N}).
\end{align}
Note that the indices of the operators $A_i$ and $B_i$ are elements of $\{1,2\}$ instead of $\{0,1\}$; thus ---for technical purposes only--- we need the following lemma:
\begin{lemma}\label{l: biject_to_binary}
	Let $\alpha\in\mathbb N$.  The map
	\[ \theta_\alpha: {\bf I}_\alpha \rightarrow \{0,1,\cdots,2^\alpha-1\};
	\quad {\bf i}=(i_1,\cdots,i_\alpha) \mapsto 1 - 2^\alpha + \sum_{l=1}^{\alpha} 2^{l-1} i_l \]
	is a well-defined bijection.
\end{lemma}

\begin{proof}
	Let ${\bf i}=(i_1,\cdots,i_\alpha) \in \mathbf{I}_{\alpha}$. Set $j_l := i_l - 1$ for each $l\in \{1, \ldots, \alpha\}$, so that $j_l\in\{0,1\}$.  Then
	\[ \theta_\alpha({\bf i}) = 1 - 2^\alpha + \sum_{l=1}^\alpha 2^{l-1} (j_l+1)
	= 1 - 2^\alpha + (2^\alpha-1) + \sum_{l=1}^\alpha 2^{l-1} j_l
	= \sum_{l=1}^\alpha 2^{l-1} j_l. \]
	This sum is indeed a member of $\{0,1,\cdots,2^\alpha-1\}$.  The result now follows by considering the binary decomposition of an arbitrary element of $\{0,1, \ldots, 2^{\alpha}-1\}$.
\end{proof}
\begin{lemma}\label{l: divisible}
	Let $n \in \mathbb{N}$ and let $\alpha \in \mathbb{N}_0$ be fixed. There exists a unique $\mathbf{i}=(i_1,i_2, \ldots, i_{\alpha}) \in \mathbf{I}_{\alpha}$ such that $n+ \sum\limits_{l=1}^{\alpha} 2^{l-1} i_l -2^{\alpha} +1$ is divisible by $2^{\alpha}$.
\end{lemma}

\begin{proof}
	Working modulo $2^\alpha$, we wish to show that there is a unique $\bf i$ with $n + \theta_\alpha({\bf i}) \equiv 0 \mod 2^\alpha$, or equivalently, $\theta_\alpha({\bf i}) \equiv -n \mod 2^\alpha$.  This follows directly from Lemma~\ref{l: biject_to_binary}.
\end{proof}

\begin{proposition}\label{e: Ak_disjoint_coranges}
	Let $\alpha\in\mathbb N_0$.  We have that
	\begin{align}
	\sum_{{\bf k} \in {\bf I}_\alpha} \| A_{\bf k^*} x \|^p = \|x\|^p    \qquad (x\in\ell^p).
	\end{align}
\end{proposition}

\begin{proof}
	If $\alpha=0$ then the claim is trivial, so assume $\alpha\geqslant 1$.  By (\ref{e: actionofaistar}), and as clearly $\rho_{\bf k} \colon \mathbb N \rightarrow \mathbb R$ is injective for each $\mathbf{k} \in \mathbf{I}_{\alpha}$, we see that
	\begin{align}
	\sum_{{\bf k} \in {\bf I}_\alpha} \| A_{\bf k^*} x \|^p
	= \sum_{{\bf k} \in {\bf I}_\alpha} \Big\| \sum_{ \{n:\rho_{\bf k}(n) \in \mathbb N\} } x(n) e_{ \rho_{\bf k}(n) } \Big\|^p
	= \sum_{{\bf k} \in {\bf I}_\alpha} \sum_{ \{n:\rho_{\bf k}(n) \in \mathbb N\} } |x(n)|^p.
	\end{align}
	By Lemma~\ref{l: divisible} we see that for each $n \in \mathbb{N}$ there is a unique ${\bf k}\in{\bf I}_\alpha$ with $\rho_{\bf k}(n) \in \mathbb N$.  Hence
	\begin{align}
	\|x\|^p = \sum\limits_{n=1}^{\infty} |x(n)|^p = \sum_{{\bf k} \in {\bf I}_\alpha} \sum_{ \{n:\rho_{\bf k}(n) \in \mathbb N\} } |x(n)|^p.
	\end{align}
	The result follows.
\end{proof}

We require one more piece of notation.  Given $\mathbf{i} \in \mathbf{I}$ we define the map $\sigma_{\mathbf{i}} \colon \mathbb{N} \to \mathbb{N}$ by
\begin{align}
\sigma_{\mathbf{i}}(n):= \begin{cases}
2^{\alpha} n - \sum\limits_{l=1}^{\alpha} 2^{l-1} i_l +2^{\alpha} -1  &\text{ if } \mathbf{i}= (i_1,i_2, \ldots i_{\alpha}) \in \mathbf{I} \setminus \{\emptyset\}, \\
n &\text{ if } \mathbf{i}= \emptyset
\end{cases} \quad (n \in \mathbb{N}).
\end{align}
An elementary induction argument on $\alpha$ and \eqref{e: actionofbi} show that
\begin{align}\label{e: actionofb}
B_{\mathbf{i}} e_n = e_{\sigma_{\mathbf{i}}(n)} \qquad (n \in \mathbb{N}).
\end{align}

\begin{proposition}\label{e: Bk_disjoint_ranges}
	Let $\alpha\in\mathbb N_0$. Then for any family of vectors $(x_{\bf k})_{{\bf k} \in {\bf I}_\alpha}$ in $\ell^p$,
	\begin{align}
	\Big\| \sum_{{\bf k} \in {\bf I}_\alpha} B_{\bf k} x_{\bf k} \Big\|^p
	= \sum_{{\bf k} \in {\bf I}_\alpha} \| x_{\bf k} \|^p.
	\end{align}
	In particular, $B_{\bf j}$ is an isometry for each ${\bf j} \in {\bf I}_\alpha$.
\end{proposition}

\begin{proof}
	The $\alpha=0$ case is trivial, so let $\alpha\geqslant 1$. We \textit{claim} that the ranges of $\sigma_{\bf j}$ and $\sigma_{\bf k}$ are disjoint for distinct $\mathbf{j}, \mathbf{k} \in \mathbf{I}_{\alpha}$. Indeed, suppose $\sigma_{\bf j}(n) = \sigma_{\bf k}(m)$ for some $n,m \in \mathbb{N}$ and $\mathbf{j}, \mathbf{k} \in \mathbf{I}_{\alpha}$. In the notation of Lemma~\ref{l: biject_to_binary}, this is equivalent to $2^\alpha n - \theta_\alpha({\bf j}) = 2^\alpha m - \theta_\alpha({\bf k})$; thus $|n-m| = 2^{- \alpha}|\theta_\alpha({\bf j}) - \theta_\alpha({\bf k})| \leqslant (2^{\alpha}-1)2^{-\alpha} <1$. Therefore $n=m$ and hence $\mathbf{j} = \mathbf{k}$ as $\theta_{\alpha}$ is bijective. From the claim and \eqref{e: actionofb} we obtain
	\begin{align}
	\Big\| \sum_{{\bf k} \in {\bf I}_\alpha} B_{\bf k} x_{\bf k} \Big\|^p = \Big\| \sum_{{\bf k} \in {\bf I}_\alpha} \sum\limits_{n=1}^{\infty} x_{\bf k}(n) e_{\sigma_{\mathbf{k}}(n)} \Big\|^p = \sum_{{\bf k} \in {\bf I}_\alpha} \sum\limits_{n=1}^{\infty} | x_{\bf k}(n) |^p = \sum_{{\bf k} \in {\bf I}_\alpha} \| x_{\bf k} \|^p,
	\end{align}
	as required. To see the ``in particular part'', let ${\bf j} \in {\bf I}_\alpha$ and $x \in \ell^p$. Set $x_{\bf j}:= x$ and $x_{\bf k}:=0$ for each ${\bf k} \in {\bf I}_\alpha \setminus \{ {\bf j} \}$, then $\|B_{{\bf j}} x \| = \|x\|$ follows from the first part.
\end{proof}

\begin{lemma}\label{l: isometry}
	Let $\alpha \in \mathbb{N}_0$ and let $\mathbf{j} \in \mathbf{I}$.  The operator $\sum_{\mathbf{k} \in \mathbf{I}_{\alpha}} B_{\mathbf{kj}} A_{\mathbf{k}^{*}}$ is an isometry on $\ell^p$.
\end{lemma}

\begin{proof}
	We combine Propositions~\ref{e: Ak_disjoint_coranges} and~\ref{e: Bk_disjoint_ranges} to see that
	\begin{align}
	\Big\| \sum_{\mathbf{k} \in \mathbf{I}_{\alpha}} B_{\mathbf{kj}} A_{\mathbf{k}^{*}} x\Big\|^p
	=\Big\| \sum_{\mathbf{k} \in \mathbf{I}_{\alpha}} B_{\mathbf{k}} (B_{\mathbf{j}} A_{\mathbf{k}^{*}} x) \Big\|^p
	= \sum_{\mathbf{k} \in \mathbf{I}_{\alpha}} \| B_{\bf j} A_{\bf k^*} x \|^p
	= \sum_{\mathbf{k} \in \mathbf{I}_{\alpha}} \| A_{\bf k^*} x \|^p
	= \|x\|^p 
	\end{align}
	for any $x\in \ell^p$.
\end{proof}

\begin{lemma}\label{l: contraction}
	Let $\alpha \in \mathbb{N}_0$ and let $\mathbf{j} \in \mathbf{I}$.  The operator $\sum_{\mathbf{k} \in \mathbf{I}_{\alpha}} B_{\mathbf{k}} A_{(\mathbf{kj})^{*}}$ has norm less than or equal to $1$.
\end{lemma}

\begin{proof}
	We combine Propositions~\ref{e: Ak_disjoint_coranges} and~\ref{e: Bk_disjoint_ranges} to see that for any $x\in \ell^p$,
	\begin{align}
	\Big\| \sum_{\mathbf{k} \in \mathbf{I}_{\alpha}} B_{\mathbf{k}} A_{(\mathbf{kj})^{*}} x \Big\|^p
	= \sum_{\mathbf{k} \in \mathbf{I}_{\alpha}} \| A_{({\bf kj})^{*}} x \|^p
	= \sum_{\mathbf{k} \in \mathbf{I}_{\alpha}} \| A_{\bf j^*} A_{\bf k^*} x \|^p
	\leqslant \sum_{\mathbf{k} \in \mathbf{I}_{\alpha}} \| A_{\bf k^*} x \|^p
	= \|x\|^p,
	\end{align}
	as $\|A_{\bf j^*}\|\leqslant 1$.
\end{proof}

\begin{theorem}\label{t: no_trace_on_o_2}
	There are no non-zero bounded traces on $\mathcal{O}_2^p$. 
\end{theorem}

\begin{proof}
	Assume towards a contradiction that $\tau'$ is a non-zero bounded trace on $\mathcal{O}_2^p$. As $\Theta_p \colon \mathcal{A} / \mathcal{J} \to \mathcal{B}(\ell^p)$ is a continuous algebra homomorphism with $\overline{\Ran}(\Theta_p)= \mathcal{O}_2^p$ by Proposition~\ref{p: rep_quot}, it follows that $\tau := \tau' \circ \Theta_p |^{\mathcal{O}_2^p}$ is a non-zero bounded trace on $\mathcal{A} / \mathcal{J}$.  By an abuse of notation, we identify $\tau$ with a member of $\ell^{\infty}(\text{Cu}_2 \setminus \{\lozenge\})$, see the discussion before Lemma~\ref{l: char_trace}.
	
	As $\tau$ is non-zero, after possibly rescaling, we obtain that there is a $w \in \text{Cu}_2 \setminus \{\lozenge\}$ such that $\tau(w) =1$. Note that there is no $u \in \text{Cu}_2 \setminus \{\lozenge\}$ such that $w=s_i u s_j^*$, where $i,j \in \{1,2\}$ are distinct. Indeed, otherwise $\tau(w)= \tau(s_i u s_j^*) = \tau(s_j^*s_i u)=0$. Also note that if $w= s_i u s_i^*$ for some $i \in \{1,2\}$, then $\tau(w)= \tau(s_i u s_i^*) = \tau(s_i^*s_i u)=\tau(u)$. Hence we may assume without loss of generality that $w= s_{\mathbf{j}}$ or $w=s_{\mathbf{j}}^*$ for some $\mathbf{j} \in \mathbf{I} \setminus \{\emptyset\}$.
	
	First suppose that $w= s_{\mathbf{j}}$.  For every $\mathbf{k} \in \mathbf{I}$,
	\begin{align}\label{e: trace_one}
	\left\langle B_{\mathbf{kj}} A_{\mathbf{k}^{*}}, \tau' \right\rangle &= \left\langle B_{\mathbf{k}} B_{\mathbf{j}} A_{\mathbf{k}^{*}}, \tau' \right\rangle = \left\langle A_{\mathbf{k}^{*}} B_{\mathbf{k}} B_{\mathbf{j}}, \tau' \right\rangle = \left\langle B_{\mathbf{j}}, \tau' \right\rangle \notag \\
	&= \left\langle \Theta(\pi_{\mathcal{J}}(\delta_{s_{\mathbf{j}}})), \tau' \right\rangle = \left\langle \pi_{\mathcal{J}} (\delta_{s_{\mathbf{j}}}), \tau \right\rangle = \tau(s_{\mathbf{j}}) =1.
	\end{align}
	Using \eqref{e: trace_one} and Lemma~\ref{l: isometry} we obtain for every $\alpha \in \mathbb{N}$,
	\begin{align}\label{e: tracecalk}
	2^{\alpha} &= \Big| \sum\limits_{\mathbf{k} \in \mathbf{I}_{\alpha}} \left\langle B_{\mathbf{kj}} A_{\mathbf{k}^{*}}, \tau' \right\rangle \Big| = \Big| \Big\langle \sum\limits_{\mathbf{k} \in \mathbf{I}_{\alpha}} B_{\mathbf{kj}} A_{\mathbf{k}^{*}}, \tau' \Big\rangle \Big| \leqslant \| \tau' \| \Big\| \sum\limits_{\mathbf{k} \in \mathbf{I}_{\alpha}} B_{\mathbf{kj}} A_{\mathbf{k}^{*}} \Big\| = \| \tau' \|,
	\end{align}
	which is nonsense.
	
	Now suppose that $w = s_{\bf j}^*$, so that
	\begin{align}
	\langle B_{\bf k} A_{({\bf kj})^*}, \tau' \rangle
	= \langle B_{\bf k} A_{\bf j^*} A_{\bf k^*}, \tau' \rangle
	= \langle A_{\bf j^*}  A_{\bf k^*} B_{\bf k}, \tau' \rangle
	= \langle A_{\bf j^*}, \tau' \rangle
	= \tau(s_{\bf j}^*) = 1,
	\end{align}
	analogously to the calculation in \eqref{e: trace_one}. Using Lemma~\ref{l: contraction} and an estimate completely similar to \eqref{e: tracecalk} this again leads to a contradiction.
	
	Hence there are no non-zero bounded traces on $\mathcal{O}_2^p$.
\end{proof}

\begin{remark}
	It is much easier to prove however the weaker statement that there are no \textit{normalised} traces on $\mathcal{O}_2^p$. Indeed, let $\tau'$ be a trace on $\mathcal{O}_2^p$. As $I_{\ell^p} = I_{\ell^p} + I_{\ell^p} - I_{\ell^p} = A_1 B_1 + A_2 B_2 - (B_1 A_1 + B_2 A_2)$, the equality $\left\langle I_{\ell^p}, \tau' \right\rangle =0$ follows; whence $\tau'$ cannot be normalised.
\end{remark}

We end the paper with a question of interest to Banach algebraists.  Motivated by \cite{p2}, we ask if $\mc A/\mc J$ is an \emph{amenable} Banach algebra?  Phillips shows that $\mc O^p_2$ is amenable, but his techniques do not appear applicable to $\mc A/\mc J$ due to differing (again, incomparable) norm estimates.  However, if $\mc A/\mc J$ were amenable, this would immediately give a new proof that $\mc O^p_2$ is amenable.

\begin{ack}
We thank Yemon Choi and Hannes Thiel for helpful remarks about the wider literature, and Yemon Choi for a suggestion about looking at traces on $\mc A/\mc J$. We are extremely grateful to the anonymous referee for the unusually detailed report and for the constructive suggestions therein. We thank them especially for also suggesting the study of traces on $\mc A / \mc J$ and $\mc O^p_2$.
The first-named author is partially supported by EPSRC grant EP/T030992/1.
The second-named author acknowledges with thanks the funding received from GA\v{C}R project 19-07129Y; RVO 67985840 (Czech Republic).
\end{ack}

\Addresses

\end{document}